\documentclass[11pt]{amsart}
\usepackage{latexsym,amsmath,amssymb,amsfonts,amscd,graphics,appendix,amsxtra}
\usepackage[mathscr]{eucal}
\usepackage{fullpage}

\usepackage{xr-hyper}
\usepackage{tikz-cd}

\usepackage{hyperref}

\newtheorem{theorem}{Theorem}[section]
\newtheorem{proposition}[theorem]{Proposition}
\newtheorem{lemma}[theorem]{Lemma}

\newtheorem{corollary}[theorem]{Corollary}

\newtheorem{D}[theorem]{Definition}
\newenvironment{definition}{\begin{D} \rm }{\end{D}}
\newtheorem{R}[theorem]{Remark}
\newenvironment{remark}{\begin{R}\rm }{\end{R}}
\newtheorem{E}[theorem]{Example}
\newenvironment{example}{\begin{E}\rm }{\end{E}}

\numberwithin{equation}{section}

\def\Zee{\mathbb{Z}}
\def\Q{\mathbb{Q}}

\def\Cee{\mathbb{C}}
\def\Pee{\mathbb{P}}

\def\Ker{\operatorname{Ker}}
\def\Coker{\operatorname{Coker}}
\def\Hom{\operatorname{Hom}}
\def\Ext{\operatorname{Ext}}
\def\Aut{\operatorname{Aut}}
\def\Pic{\operatorname{Pic}}
\def\Gr{\operatorname{Gr}}

\def\im{\operatorname{Im}}

\def\Spec{\operatorname{Spec}}

\def\scrO{\mathcal{O}}
\def\spcheck{^{\vee}}
\def\hX{\widehat{X}}
\def\hY{\widehat{Y}}

\title{Deformations of   Calabi-Yau varieties with isolated  log canonical singularities} 

\begin{document}
\author[R. Friedman]{Robert Friedman}
\address{Columbia University, Department of Mathematics, New York, NY 10027}
\email{rf@math.columbia.edu}
\author[R. Laza]{Radu Laza}
\address{Stony Brook University, Department of Mathematics, Stony Brook, NY 11794}
\email{radu.laza@stonybrook.edu}

\begin{abstract}
Recent progress in the  deformation theory of Calabi-Yau varieties $Y$ with canonical singularities has highlighted the key role played by the higher Du Bois and higher rational singularities, and especially by the so-called  $k$-liminal singularities  for $k\ge 1$. The goal of this paper is to show  that certain aspects of this study extend naturally to the $0$-liminal case as well, i.e. to Calabi-Yau varieties $Y$ with Gorenstein log canonical, but not canonical, singularities.  In particular, we show the existence of first order smoothings of $Y$ in the case of isolated $0$-liminal hypersurface singularities, and extend Namikawa's unobstructedness theorem for deformations of singular Calabi-Yau threefolds $Y$ with canonical singularities   to the case  where $Y$ has an isolated $0$-liminal lci singularity under suitable hypotheses. Finally, we describe an interesting series of examples. 
\end{abstract}
\thanks{Research of the second author is supported in part by NSF grant DMS-2101640.}  

\bibliographystyle{amsalpha}
\maketitle

\section*{Introduction}

A \textsl{generalized Calabi-Yau variety} is a  compact  analytic space  $Y$ with isolated   Gorenstein  singularities and trivial dualizing sheaf. In case the singularities are in addition  rational  (or equivalently canonical), their deformation theory has been studied extensively, beginning with the pioneering work of Burns-Wahl on $K3$ surfaces with rational double points \cite{BurnsWahl}. In dimension $3$, Namikawa and Steenbrink studied the smoothability of such varieties in a series of papers \cite{namtop}, \cite{NS}, \cite{namstrata}. More recently, the authors have revisited and extended  their results and partially generalized them to higher dimensions \cite{FL}. A crucial new ingredient is the notion of higher Du Bois and higher rational singularities. For a $k$-Du Bois or $k$-rational singularity, the integer  $k$ is a measure  of the complexity of the singularity: as  $k$ increases, the singularity becomes simpler in the sense that, once $k$ reaches a certain threshold, the singularity is  an ordinary double point, and for larger $k$   it is a  smooth point. As an example of some of our recent results,  it follows from  Corollary 1.5 in \cite{FL22c} that, if $Y$ has local complete intersection $1$-Du Bois singularities, then its deformations are unobstructed. By contrast, if $Y$ has rational hypersurface singularities  which are \textbf{not} $1$-rational, then \cite[Theorem 5.8]{FL} gives necessary and sufficient conditions for the existence of  first order smoothings under some mild technical assumptions.  Somewhat surprisingly, it is thus easier to prove results about the existence of first order smoothings for more complicated singularities. Note that there is only a small overlap in the  two results above: they apply simultaneously only to  singularities which are  $1$-Du Bois but not $1$-rational.  Such singularities are defined to be \textsl{$1$-liminal} \cite{FL22d}. For example,  ordinary double points in dimension $n> 3$ are $1$-rational; and hence $1$-Du Bois, but they are not $1$-liminal.    Generalizing a result of Rollenske-Thomas  \cite{RollenskeThomas},  we showed in   \cite{FL23} that, for $k\ge 1$,  interesting deformation theoretic phenomena occur  for $Y$  which are \textsl{$k$-liminal} in the terminology of \cite{FL22d}, i.e.\ $k$-Du Bois but not $k$-rational, which includes the case of ordinary double points in odd dimension studied in \cite{RollenskeThomas}. 

Unfortunately, generalized Calabi-Yau varieties with  rational  singularities are birationally at least as complicated as smooth Calabi-Yau varieties and it is often quite difficult to construct them.  In dimension $2$, $K3$ surfaces with rational double points can be constructed via the period map. In   dimension $3$, smooth rational curves on a Calabi-Yau threefold with normal bundle $\scrO(-1) \oplus\scrO(-1)$ can be contracted to ordinary double points. In general, however, aside from some several ingenious but \emph{ad hoc} methods for constructing hypersurfaces with many ordinary double points, there seems to be no general method for constructing such varieties. Thus, in this paper, we go in the other direction from \cite{RollenskeThomas}, \cite{FL}, \cite{FL23} and consider  the boundary case where $Y$ is  assumed to have isolated Gorenstein Du Bois but nonrational singularities.  By a theorem of Ishii \cite{Ishii85}, these singularities are log canonical but not canonical. In the terminology of \cite{FL22d}, these singularities are $0$-Du Bois but not $0$-rational, or equivalently  are $0$-liminal.  In addition, we impose a few extra technical conditions, leading to what we call a \textsl{strict $0$-liminal  Calabi-Yau variety} (Definition~\ref{definestrict}).  Such varieties tend to be birationally much simpler than Calabi-Yau varieties. In particular, a  resolution of singularities of such a variety  tends to be rational or at least rationally connected. 

In dimension $2$, normal Gorenstein singularities which are Du Bois  but nonrational   are either simple elliptic or cusp singularities \cite[Theorem 3.8]{Steenbrink}, and the local and global aspects of their deformation theory  have   been extensively studied. For example, by a theorem of Wahl \cite[Theorem 5.6]{Wahl81}, if $Y$ is a compact analytic surface with a unique isolated singular point which is a simple elliptic or cusp singularity and $\omega_Y \cong \scrO_Y$, then $H^2(Y; T^0_Y) = 0$, where $T^0_Y$ is the tangent sheaf of $Y$, which implies that  the natural map $\mathbb{T}^1_Y = \Ext^1(\Omega^1_Y, \scrO_Y) \to H^0(Y; T^1_Y)$ is surjective. Moreover,  the deformation functor of $Y$ is smooth over the functor of local deformations of the singular point. If  the singular point is a local complete intersection, then $\mathbb{T}^2_Y =0$, where $\mathbb{T}^2_Y = \Ext^2(\Omega^1_Y, \scrO_Y)$ is the obstruction space for the deformation functor, so that the deformations of $Y$ are unobstructed.  More generally,  if the singularity is locally smoothable, for example if it is a local complete intersection singularity, then $Y$ is smoothable, and all smoothings will be $K3$ surfaces. In this way,  the local  theory of deformations of simple elliptic and cusp singularities is closely connected to the study of degenerations of $K3$ surfaces. 
In the other direction, one main focus of the study of simple elliptic or cusp singularities was to find necessary and sufficient conditions for (local) smoothability in case the singularity is not a local complete intersection, cf.\ for example Pinkham \cite{Pinkham}, Looijenga \cite{Looijenga}, and more recently \cite{GHK}, \cite{Engel}. Roughly speaking,  the idea is to find a compact complex surface $Z$ with an isolated singular point analytically isomorphic to the relevant simple elliptic or cusp singularity and to show that local smoothings of the singular point are equivalent to smoothings of  $Z$   to a del Pezzo surface or more generally a smooth rational surface with an effective anticanonical divisor.  In particular, if the local singularity is smoothable, then  $Z$ can be smoothed to a log Calabi-Yau surface.

In dimension $2$, a smoothing of a generalized Calabi-Yau surface with rational double points has finite monodromy, whereas smoothings of simple elliptic and cusp singularities have infinite monodromy.  In higher dimensions, a  smoothing of a generalized Calabi-Yau variety with   rational  singularities  can have infinite monodromy, although the monodromy will never be maximal unipotent (cf.\ \cite[Remark 4.11]{FL23}). Moreover, such varieties are   at finite distance in the Weil-Petersen metric in an appropriate sense (\cite{tosatti}, \cite{wang}, \cite{KLSV}, as well as \cite[Theorem 2.2]{Odaka-CY}, \cite{yzhang}). Thus, in some sense such varieties lie in the interior of the moduli space. By contrast, in the nonrational Du Bois case, the monodromy is never finite but is not necessarily maximal unipotent (cf.\ Remark~\ref{infmono}). One interpretation of this fact is that strict $0$-liminal  Calabi-Yau varieties lie at the boundary of the moduli space of Calabi-Yau manifolds. In terms of the minimal model program,  strict $0$-liminal  Calabi-Yau varieties are not dlt and  typically one has to pass to a semistable model to insure the dlt condition. For the special case of $d$-semistable normal crossing varieties, the deformation theory is understood by \cite{KawamataNamikawa}. We make a few comments about smoothing a   strict $0$-liminal  Calabi-Yau variety  with an isolated $0$-liminal singularity versus smoothing the normal crossing model in a very special situation (see \S\ref{section4}).

Our first goal of   this paper is to generalize some of the results of \cite{FL} to the $0$-liminal case, i.e.\ the case where $Y$ has log canonical but not canonical singularities. For such $Y$, we will always denote by $\pi\colon \hY \to Y$  a log resolution of $Y$ with exceptional divisor $E = \bigcup_iE_i$, where the $E_i$ are the components of $E$. The starting assumption is that $Y$ has isolated locally complete intersection (lci) $0$-liminal singularities. This is needed to guarantee that there is a good local deformation theory (in particular the singularities of $Y$ are locally smoothable) and that the techniques involving higher Du Bois and higher rational singularities apply.  While our main concern in this paper is with the lci case,   a recent paper of Coughlan-Sano \cite{CS} gives examples of  cones over $K3$ surfaces which are not (locally) smoothable, and we briefly describe the relation between some of these examples and  our point of view in \S\ref{section2.5}.  

Under the appropriate hypotheses, many of the local and global results in \cite{FL} carry over to the $0$-liminal case. For example, illustrating the principle that it is easier globally to find first order smoothings for more complicated singularities, we are able to prove results about the existence of first order smoothings in the  case of hypersurface singularities (Theorem~\ref{maintheorem} and Theorem~\ref{smoothable}). In particular, we show the following:
\begin{theorem}\label{thm1}
Let $Y$ be a   strict $0$-liminal  Calabi-Yau variety (Definition~\ref{definestrict}) whose unique singular point is a hypersurface singularity. Suppose that either: {\rm(i)} the singularity is   weighted homogeneous   or {\rm(ii)} $Y$ is $3$-dimensional and $H^2(E_i; \scrO_{E_i}) = 0$ for all $i$. Then $Y$ has a first order smoothing (Definition~\ref{defstsm}). 
\end{theorem}

In higher dimensions, it is not realistic to expect that $\mathbb{T}^2_Y =0$. However, in dimension $3$, Namikawa proved an unobstructedness result for Calabi-Yau threefolds with isolated lci rational singularities.  We are able to establish  an analogue of Namikawa's result   in the  isolated lci $0$-liminal  case, again in dimension $3$,  under some mild extra hypotheses (Theorem~\ref{dim3unob}):

\begin{theorem}\label{thm2} Let $Y$ be a strict $0$-liminal Calabi-Yau threefold. Assume that $\Aut Y$ is discrete. Then the deformations of $Y$ are unobstructed.
\end{theorem}

In particular, in dimension $3$, combining Theorems \ref{thm2} and \ref{thm1}, we obtain the following generalization of the deformation theory of cusp and simple elliptic surface singularities: 

\begin{corollary}
Let $Y$ be a strict $0$-liminal Calabi-Yau threefold whose  unique singular point is a hypersurface singularity, and which is either weighted homogeneous  or for which $H^2(E_i; \scrO_{E_i}) = 0$ for all $i$. Assume additionally that $\Aut Y$ is discrete. Then $Y$ is smoothable. 
\end{corollary}

 In dimension greater than $3$, there is as yet no known unobstructedness result for  deformations of Calabi-Yau varieties with rational singularities which are not $1$-Du Bois, and the corresponding result in the $0$-liminal case  is similarly unknown and is likely to be as least as difficult. Nonetheless, the results of Kawamata-Namikawa \cite{KawamataNamikawa} suggest that one might expect such results to hold. One could also ask if, by analogy with the situation in dimension $2$, there are reasonably general hypotheses  under which the deformation functor of $Y$ is smooth over the local deformation functor of the singularity. However, we show in Example~\ref{finalexample} that this can fail in dimension at least $3$, even for very simple examples. Similar constructions work for generalized Fano and Calabi-Yau varieties with rational singularities in dimension at least $4$ as noted in \cite[Remark 4.12(v)]{FL}.

The contents of this paper are as follows. We start in \S\ref{section1} by a preliminary discussion of the higher dimensional analogues of the simple elliptic and cusp  singularities  (i.e. the $0$-liminal singularities in dimension $2$). This is inspired by the work of Ishii \cite{Ishii85} (see also \cite{Fujino2011}, \cite{Ishii2012}, and \cite{Yonemura}). In Section \ref{section2},  we adapt the arguments of \cite{FL} to the case of $0$-liminal singularities. The key technical result is Theorem \ref{maintheorem}, which gives a filtration on the local deformation space of $0$-liminal singularity in terms of Hodge theoretic (``Du Bois" type) data. A case where everything becomes more transparent and computable is the weighted homogeneous case discussed in Proposition \ref{wtdcase}. In Section \ref{section3}, after some general considerations, we then prove our two main results Theorem \ref{dim3unob} (unobstructedness in dimension $3$) and Theorem \ref{smoothable} (existence of first order smoothings). In Section~\ref{section4}, we make some remarks about the deformations of the pair $(\hY,E)$, where $\hY \to Y$ is a log resolution of singularities,  under the special assumption that $K_{\hY} = \scrO_{\hY}(-E)$. In addition, we compare these deformations to the  locally trivial deformations of a normal crossing model under very restrictive hypotheses. 
In Section~\ref{section5}, we discuss in detail what are perhaps the simplest examples of our  theory,  hypersurfaces $Y$ of degree $n+2$ in $\Pee^{n+1}$ with an ordinary $(n+1)$-fold point.  Note that, for $n=2$, the surfaces $Y$ are quartic surfaces with an ordinary triple point and as such are perhaps the simplest degenerations of a $K3$ surface to a normal surface with nonrational singularities that one might encounter. By contrast to the case $n=2$, we show that, for $n \ge 3$,   the deformation functor of $Y$ is never smooth over the local deformation functor of the singularity (Example~\ref{finalexample}) and compare the deformations of $Y$ with the deformations of the corresponding $d$-semistable models. While the methods of this section are elementary, they serve to illustrate the general theory and are a concrete manifestation of how the study of $0$-liminal  Calabi-Yau varieties of dimension $n$  is connected with Fano varieties of dimension $n$ and (smooth) Calabi-Yau varieties of dimension $n-1$ for every $n \ge 2$.  

This paper should be viewed as a sequel to  \cite{FL}, \cite{FL22b}, and \cite{FL23}. Together, these  papers hint at a very rich picture in the   study of deformations of singular Calabi-Yau varieties, where the Hodge theory of singularities plays a central role.  However, many open questions remain. For example, as noted above, it is important to establish more general unobstructedness results in dimensions greater than $3$. Despite the failure of the natural morphism  from the deformation functor of a (global) singular Calabi-Yau variety to the (local) deformations of the singular point to be smooth for dimension at least $3$, one could also ask if there is a natural subfunctor of the local deformation functor over which the global deformation functor  is smooth. Roughly speaking this is the approach of Kawamata-Namikawa \cite{KawamataNamikawa} in the normal crossings case, and such an approach might apply  more generally  for   isolated singularities without the  lci assumption. Finally, even for the case of canonical hypersurface singularities, while the methods of \cite{FL23} find necessary conditions for the existence of first order smoothings for $k$-liminal singularities, finding  sufficient conditions for first order smoothings to exist is completely mysterious for $k > 1$ as well as what happens when the singularities are $1$-rational but not $k$-liminal (for example, ordinary double points in even dimension at least $4$).
 
\subsection*{Acknowledgements}   We would  like to thank the referees for a careful reading of the paper and several helpful suggestions which helped to improve it.  We also thank  Rosie Shen, who raised a question at a recent  AIM workshop on higher Du Bois and higher rational singularities concerning the  local smoothability for cones over $K3$ surfaces which led to the discussion in \S\ref{section2.5}, and to AIM for hosting the workshop.

\section{The local case}\label{section1} 

In this section, we review some basic properties of isolated log canonical Gorenstein singularities, first discussed by Ishii \cite{Ishii85}. The two basic examples of such singularities are (i) the singularities which admit a good anticanonical resolution (Definition~\ref{def-goodanti}); these are natural generalizations of the simple elliptic and cusp singularities from dimension $2$, and (ii) weighted homogeneous singularities satisfying an appropriate numerical condition \eqref{eq-0lim};  such singularities should be viewed as essentially cones over  Calabi-Yau manifolds of dimension $n-1$.

\subsection{Basic definitions and results} Throughout this section, unless otherwise noted, $(X,x)$ denotes the germ of an isolated Gorenstein singularity of dimension $n \ge 3$. We will also use $X$ to denote a contractible Stein representative for the germ $(X,x)$. We assume throughout that  $\pi\colon \hX \to X$ is  a \textsl{good or log resolution}, i.e.\ that  the exceptional divisor $\pi^{-1}(x) = E=\bigcup_iE_i$ has simple normal crossings, and that in addition it is an \textsl{equivariant   resolution}, i.e.\ that $\pi_*T_{\hX} = T^0_X$.  (See for example \cite[1.11, 1.12]{FL} for this terminology and the argument for the existence of  good equivariant   resolutions.)  Let $U = X -\{x\}= \hX -E$.  Note that $E$ is a deformation retract of $\hX$.  We  begin by recalling some of the basic numerical invariants of $(X,x)$, and refer to \cite{SteenbrinkDB}, \cite[\S1]{FL},   \cite{FL22d} for more details.

\begin{definition}\label{defineinv}   Define the \textsl{Du Bois invariants} $b^{p,q} = \dim H^q(\hX; \Omega^p_{\hX}(\log E)(-E))$ (for $q > 0$) and  the \textsl{link invariants} $\ell^{p,q} = \dim H^q(\hX; \Omega^p_{\hX}(\log E)|E)$. Here, if $L$ denotes the link of the isolated singularity $X$, then the cohomology groups $H^k(L)$ are naturally mixed Hodge structures and $\ell^{p,q} = \dim \Gr^p_FH^{p+q}(L)$.  If $X$ is in addition a local complete intersection of dimension $n$, define
$$s_p = \dim \Gr^p_FH^n(M),$$ 
where $M$ is the Milnor fiber and $F$ is the Hodge filtration on $H^n(M)$. 
\end{definition}

The following result is well-known \cite[Lemma 2.3(iv)]{FL22d}. 

\begin{proposition}\label{linkexseq} There is a long  exact sequence of mixed Hodge structures, the link exact sequence:
$$\cdots \to H^k(E) \to H^k(L) \to H^{k+1}_E(\hX) \to  \cdots.$$
It is the long exact hypercohomology sequence associated to the exact sequence of complexes
$$0 \to  \Omega^\bullet_E/\tau_E^\bullet\to  \Omega^\bullet_{\hX}(\log E)|E \to  \Omega^\bullet_{\hX}(\log E)/\Omega^\bullet_{\hX} \to  0.$$
Here $\tau_E^\bullet$ is the subcomplex of differentials on $E$ supported on $E_{\text{\rm{sing}}}$, $\mathbb{H}^k(E; \Omega^\bullet_E/\tau_E^\bullet) = H^k(E)$, $\mathbb{H}^k(E;\Omega^\bullet_{\hX}(\log E)|E) = H^k(L)$ and $\mathbb{H}^k(E;\Omega^\bullet_{\hX}(\log E)/\Omega^\bullet_{\hX}) = H^{k+1}_E(\hX)$. Moreover, for each of the hypercohomology groups above, the first hypercohomology spectral sequence degenerates at $E_1$ and the corresponding filtration on cohomology is the Hodge filtration. \qed
\end{proposition}

The following lists  some of the basic vanishing results concerning these invariants:

\begin{proposition}\label{locinvvan}  With notation as in Definition~\ref{defineinv},
\begin{enumerate} 
\item[\rm(i)] For $p+q > n$, $b^{p,q} = \dim H^q(\hX; \Omega^p_{\hX}(\log E)(-E))=0$.
\item[\rm(ii)] If $X$ is a local complete intersection, then $b^{p,q} = 0$ for all $q> 0$ unless $p+q= n-1$ or $p+q=n$. 
\item[\rm(iii)] {\rm(Semipurity)} The morphism  $H^k(E) \to H^k(L)$   is surjective for 
   $k\le n$ and the morphism $H^{k-1}(L) \to  H^k_E(\hX)$ is injective for $k> n$. Thus the map $H^n_E(\hX) \to H^n(E)$ is an isomorphism. 
\item[\rm(iv)]  If $X$ is a local complete intersection, then $\ell^{p,q} = 0$   except in the following cases: $p=q=0$;  $p=n$ and $q=n-1$;    $p+q= n-1$;  $p+q=n$. 
\end{enumerate}
\end{proposition}
\begin{proof} (i) This is the vanishing theorem  of Guill\'en, Navarro Aznar, Pascual-Gainza, Puerta and Steenbrink (see e.g.\ \cite[p. 181]{PS}).

\smallskip
\noindent (ii) This is a result of Steenbrink \cite[Theorem 5]{SteenbrinkDB}.

\smallskip
\noindent (iii) This is a direct consequence of the Goresky-MacPherson semipurity theorem \cite[Corollary 1.12]{Steenbrink}. 

\smallskip
\noindent (iv) This is a consequence of the fact that, if $X$ is an isolated local complete intersection, then  $H^k(L) \neq 0$ only for $k = 0, n-1, n, 2n-1$.
\end{proof} 

Then there is the following basic definition due to Steenbrink \cite[\S3]{Steenbrink} for the case $m=0$ and to \cite{MOPW} and \cite{JKSY-duBois} in general: 

\begin{definition}\label{defkDB}  The germ of a  complex space $(X,x)$ is \textsl{Du Bois}  if the natural map $\scrO_X \to \underline\Omega^0_X$ is an isomorphism, where $\underline\Omega^0_X$ is the $0^{\text{th}}$ graded piece of the Deligne-Du Bois complex as defined in \cite{duBois}.  More generally, it is \textsl{$k$-Du Bois} if the natural map $\Omega^p_X \to \underline\Omega^p_X$ is an isomorphism for all $p\le k$, where $\underline\Omega^p_X$ is the $p^{\text{th}}$ graded piece of the Deligne-Du Bois complex.
\end{definition} 

We will also need  the following result   \cite[Theorem 5.2(iv)]{FL22d}:

\begin{proposition}\label{charmDB} An isolated local complete intersection singularity $X$ of dimension at least $3$ is $k$-Du Bois $\iff$ $X$  is $(k-1)$-Du Bois and $b^{k, n-k-1} =0$. \qed
\end{proposition}

We next recall some general results about isolated Cohen-Macaulay Du Bois singularities. 

\begin{lemma}\label{lemma1.1}  Let $(X,x)$ be a normal isolated singularity. 
\begin{enumerate} \item[\rm(i)] The singularity  $X$ is Cohen-Macaulay $\iff$  $H^i(\hX; \scrO_{\hX}) = 0$ for $0< i< n-1$. 
\item[\rm(ii)]   $X$ is Du Bois $\iff$ the natural map $H^i(\hX; \scrO_{\hX}) \to H^i(E; \scrO_E)$ is an isomorphism for all $i > 0$ $\iff$ $H^i(\hX; \scrO_{\hX}(-E))=0$ for $i> 0$, or equivalently $b^{0,i} = 0$ for all $i>0$. Hence, if $X$ is also Cohen-Macaulay, then   $H^i(E; \scrO_E) =0$ for $0< i< n-1$. 
\end{enumerate}
\end{lemma}
\begin{proof} (i): This is well-known:  The main point is that, since $X$ has an isolated singularity, $X$ is  Cohen-Macaulay $\iff$    $X$ has depth $n$ $\iff$  $H^i_x(X; \scrO_X) =0$ for $i < n$ $\iff$  $H^i(U; \scrO_U) = 0$ for $0 < i < n-1$.  By the Grauert-Riemenschneider theorem, $H^i(\hX; K_{\hX}) = 0$ for $i > 0$ and hence by duality $H^i_E(\hX; \scrO_{\hX}) = 0$ for $i < n$. Thus, for $0< i< n-1$, $H^i(\hX; \scrO_{\hX}) \cong  H^i(U; \scrO_U)$, and so $X$ is  Cohen-Macaulay $\iff$ $H^i(\hX; \scrO_{\hX}) = 0$ for $0< i< n-1$. 

\smallskip
\noindent (ii): By \cite[(3.6)]{Steenbrink}, $X$ is  Du Bois $\iff$ the natural map $H^i(\hX; \scrO_{\hX}) \to H^i(E; \scrO_E)$ is an isomorphism for all $i > 0$. From the exact sequence
$$0 \to \scrO_{\hX}(-E) \to \scrO_{\hX} \to \scrO_E \to 0, $$
this is equivalent to the vanishing of $H^i(\hX; \scrO_{\hX}(-E))$ for $i> 0$. 
 The second statement in (ii) is then a consequence of (i). 
\end{proof}

For Du Bois singularities, there is the ``extra vanishing lemma" due to Steenbrink \cite[p.\ 1369]{SteenbrinkDB}:

\begin{lemma}[Extra vanishing lemma]\label{extravan}  If $X$ is  Du Bois, then $H^{n- 1}(\hX; \Omega^1_{\hX}(\log E)(-E)) = 0$. 
Dually, $H^1_E(\hX; \Omega^{n-1}_{\hX}(\log E)) =0$. 
\qed
\end{lemma}

\subsection{Definition of $0$-liminal singularities} As the case of rational singularities is extensively discussed in \cite{FL},  we make the following definition as  in \cite{FL22d}:

\begin{definition} The isolated Gorenstein singularity $(X,x)$ is \textsl{$0$-liminal} (\textsl{purely elliptic} in the terminology of \cite{Ishii85},  or \textsl{strictly log canonical}) if it is Du Bois but not rational, or equivalently log canonical but not canonical. 
\end{definition}

The following is due to Ishii \cite{Ishii85} (Theorem 2.3 and its proof):

\begin{theorem}\label{Thm1.1}  Let $X$ be an isolated Gorenstein singularity. Then 
 $X$ is Du Bois $\iff$ $X$ is log canonical. 
 If $X$ is $0$-liminal, i.e.\ Du Bois but not rational, then $p_g(X)  = \dim H^{n-1}(\hX; \scrO_{\hX}) = \dim  H^{n-1}(E; \scrO_E) =1$.  \qed
\end{theorem} 

\begin{remark} In fact, in \cite[Theorem 2.3]{Ishii85}, it is shown  that, if $X$ is an isolated Gorenstein Du Bois singularity, then  the canonical bundle of $\hX$ satisfies:
$$K_{\hX} = \scrO_{\hX}(D - E_{\text{\rm{ess}}})$$
where $D$ is effective and $E_{\text{\rm{ess}}}\subseteq E$ is a reduced divisor (necessarily with normal crossings), the \textsl{essential divisor} of $\pi$. In particular, there is an isomorphism 
$$T_{\hX}\cong  \Omega^{n-1}_{\hX}\otimes K_{\hX}^{-1}  = \Omega^{n-1}_{\hX}\otimes \scrO_{\hX}(-D + E_{\text{\rm{ess}}})$$ and hence  there are inclusions
$$T_{\hX}(-E) \subseteq T_{\hX}(-E_{\text{\rm{ess}}}) \cong \Omega^{n-1}_{\hX} (-D) \subseteq \Omega^{n-1}_{\hX} .$$
\end{remark}

\begin{remark} If $X$ is $0$-liminal, there is a natural map $\pi^*(\mathfrak{m}_x\omega_X )\to K_{\hX}$, hence there is a sequence of inclusions
$$\mathfrak{m}_x\omega_X \subseteq \pi_* K_{\hX} \subseteq \omega_X.$$
Also, the case $\pi_* K_{\hX} = \omega_X$ is not possible, since otherwise the isomorphism $\omega_X \to \pi_* K_{\hX}$ would yield an injection $\pi^*\omega_X = \scrO_{\hX} \to  K_{\hX}$ and the singularity would be canonical, hence rational. Thus $\pi_* K_{\hX} = \mathfrak{m}_x\omega_X$. 
\end{remark}

\begin{remark}\label{infmono}  In terms of the link invariants $\ell^{p,q}$, an isolated Gorenstein  Du Bois singularity is rational $\iff$  $\ell^{0,n-1} =   0$ $\iff$ $\ell^{n,0} = 0$, and, if it  is not rational, then $\ell^{0,n-1} =   \ell^{n,0} = 1$.   In the case of a local complete intersection, an isolated  Du Bois singularity is rational $\iff$ $s_n =0$  and, if it  is not rational, then $s_n =1$, by a result of Steenbrink \cite[Proposition 2.13]{Steenbrink} that $s_n = \dim H^{n-1}(\hX; \scrO_{\hX}) $. Since $s_0 =0$ for every Du Bois singularity (cf.\ \cite[p.\ 1372]{SteenbrinkDB}), it follows that, for an isolated  $0$-liminal lci singularity,  since $s_n \neq s_0$, the mixed Hodge structure on $H^n(M)$ is never pure. In particular, the monodromy weight filtration on the limiting mixed Hodge structure of any global smoothing $\mathcal{Y} \to \Delta$  of a generalized Calabi-Yau variety which has  an isolated  $0$-liminal lci singularity is nontrivial. As the monodromy weight filtration is defined by the nilpotent operator $N$, where $N=\log T_u$ is the logarithm of the unipotent part of the monodromy operator $T$,  $T_u$ and hence $T$ have infinite order. A similar result applies in  case   $Y$ has a general isolated Gorenstein $0$-liminal singularity, not necessarily lci.
\end{remark}

One class of examples of $0$-liminal singularities is given by the following definition:

\begin{definition}\label{def-goodanti} The singularity $X$ has a \textsl{good anticanonical resolution} if there exists a good (i.e.\ log) resolution as above with $K_{\hX} = \scrO_{\hX}(-E)$. (This is closely related to Ishii's definition of an ``essential resolution" \cite[Definition 3.5]{Ishii85}  which combines aspects of this definition with crepant resolutions.) In this case, by adjunction,  $\omega_E = K_{\hX} \otimes  \scrO_{\hX}(E)|E =\scrO_E$. 
\end{definition}  

\begin{remark} In dimension $2$, simple elliptic and cusp singularities always have a good anticanonical resolution (just as rational double points always have a crepant resolution). However, this is no longer true for $0$-liminal singularities in dimension at least $3$ (just as rational Gorenstein  singularities in dimension at least $3$ need not have a crepant resolution). 
\end{remark} 

\begin{lemma}\label{goodisstrict} If $X$ has a good anticanonical resolution, then $X$ is $0$-liminal.
\end{lemma} 
\begin{proof} By the  Grauert-Riemenschneider theorem,  for $i>0$, $H^i(\hX; \scrO_{\hX}(-E)) = H^i(\hX; K_{\hX}) = 0$. Thus $X$ is Du Bois and $H^{n-1}(\hX; \scrO_{\hX}) \cong H^{n-1}(E; \scrO_E) $.  Since $\omega_E \cong\scrO_E$, $H^{n-1}(E; \scrO_E) \cong H^{n-1}(E; \omega_E) $ is Serre dual to $H^0(E;\scrO_E)$ and thus has dimension $1$. In particular, $X$ is not rational, so it is $0$-liminal.
\end{proof} 

The next lemma shows that, in many cases, it is easy to check that a singularity with $K_{\hX} = \scrO_{\hX}(-E)$ has good local properties: 

\begin{lemma}\label{newlemma1.16}  {\rm(i)} If $X$  is  an isolated normal  Cohen-Macaulay singularity such that  $K_{\hX} = \scrO_{\hX}(-E)$, then $X$ is Gorenstein.

\smallskip
\noindent {\rm(i)}  If $X$  is  an isolated normal  singularity such that  $K_{\hX} = \scrO_{\hX}(-E)$, where $E$ is smooth, $h^i(E;\scrO_E) =0$ for $1\le i \le n-2$ (i.e.\ $E$ is a strict Calabi-Yau manifold)  and  $\scrO_{\hX}(-E)|E$ is nef and big, then $X$ is Cohen-Macaulay and hence Gorenstein. 
\end{lemma}
\begin{proof} (i)  Under these hypotheses,  $K_U = K_{\hX}|\hX-E  \cong \scrO_U$ and thus $\omega_X = i_*K_U \cong i_*\scrO_U = \scrO_X$.

\smallskip
\noindent (ii) By Lemma~\ref{lemma1.1}, it suffices to show that $H^i(\hX; \scrO_{\hX}) =0$ for $1\le i \le n-2$, or equivalently that $R^i\pi_*\scrO_{\hX} =0$ for $1\le i \le n-2$. Using the formal functions theorem and the hypothesis that $h^i(E;\scrO_E) =0$ for $1\le i \le n-2$, this is an easy consequence of the fact that, by the Kawamata-Viehweg vanishing theorem, $H^i(E; \scrO_E(-nE)) =0$ for all $i > 0$. 
\end{proof} 

\begin{remark}  By Fujino \cite[Theorem 2.8]{Fujino2011}, if $(X,x)$ is Gorenstein and $0$-liminal, then there exists a partial resolution $\hX\to X$ with exceptional divisor $E$ such that $(\hX,E)$ is a $\Q$-factorial dlt pair with $K_{\hX} = \scrO_{\hX}(-E)$. The case of a good anticanonical resolution corresponds precisely to the special case when $\hX$ is smooth and $E$ has simple normal crossings. 
\end{remark}

\begin{example}\label{95} Another set of examples is given by weighted homogeneous hypersurface singularities $f=0$ of the appropriate weights. This case is formally similar to that of an $X$  with a good anticanonical resolution with a smooth exceptional divisor. Here, using the usual notation that $\Cee^*$ acts on $\Cee^{n+1}$ with weights $a_i$, $d$ is the weighted degree of $f \in \Cee[z_1, \dots, z_{n+1}]$, and $w_i = a_i/d$, the $0$-liminal condition is that $\sum_iw_i = 1$, or equivalently that  \begin{equation}\label{eq-0lim}
N = \sum_ia_i - d= d(\sum_iw_i -1) = 0.
\end{equation}
 If $\underline{E}$ is the orbifold stack associated to $E$, the numerical condition $N=0$  is equivalent to the condition $K_{\underline{E}} = \scrO_{\underline{E}}$.  For example, in dimension $3$, Reid \cite{Reid79} found $95$  families of  weighted homogeneous hypersurface singularities, mostly not in diagonal form (see also Yonemura \cite{Yonemura},  Iano-Fletcher \cite{Iano-Fletcher}). Of these, there are $14$ cases which are defined by a  deformation of a diagonal equation $z_1^p + z_2^q + z_3 ^r + z_4^s=0$, subject to the condition
$$\frac{1}{p} + \frac{1}{q} + \frac{1}{r} + \frac{1}{s} = 1.$$
\end{example}

\subsection{The exceptional divisor}\label{The exceptional divisor}  Next we discuss the exceptional divisor $E= \bigcup_iE_i$ in more detail. For a subset $I = \{i_1< i_2 < \cdots < i_r\}$, let $E_I = \bigcap_{i\in I}E_i$ be the $r$-fold intersection (if nonempty). We have the Mayer-Vietoris spectral sequence 
$$E_1^{p,q} = \bigoplus_{\#(I) = p+1}H^q(E_I; \scrO_{E_I}) \implies H^{p+q}(E; \scrO_E),$$
with $E_1$ page of the form 

\bigskip

\begin{center}
\begin{tabular}{|c|c|c|c}
$\bigoplus_iH^{n-1}(E_i; \scrO_{E_i})$ &  &  &  \\ \hline
$\bigoplus_iH^{n-2}(E_i; \scrO_{E_i})$ & $\bigoplus_{\#(I) = 2}H^{n-2}(E_I; \scrO_{E_I})$ &   &{} \\ \hline
$\bigoplus_iH^{n-3}(E_i; \scrO_{E_i})$ & $\bigoplus_{\#(I) = 2}H^{n-3}(E_I; \scrO_{E_I})$ &  $\bigoplus_{\#(I) = 3}H^{n-3}(E_I; \scrO_{E_I})$ &{} \\ \hline
\vdots &  \vdots  & \vdots  &\vdots \\ \hline
$\bigoplus_iH^0(E_i; \scrO_{E_i})$ & $\bigoplus_{\#(I) = 2}H^0(E_I; \scrO_{E_I})$ &  $\bigoplus_{\#(I) = 3}H^0(E_I; \scrO_{E_I})$ &  \dots \\ \hline
\end{tabular}

\end{center}

\bigskip

By inspection, we have:

\begin{lemma}\label{lemma1.8}   For a general isolated singularity, the map $H^{n-1}(E; \scrO_E) \to \bigoplus_iH^{n-1}(E_i; \scrO_{E_i})$ is surjective. Hence, if  $X$ is $0$-liminal, then $\bigoplus_iH^{n-1}(E_i; \scrO_{E_i})$ has dimension $0$ or $1$. Moreover, $\bigoplus_iH^{n-1}(E_i; \scrO_{E_i})$ has dimension $1$ $\iff$ there exists an $i$ such that $\dim H^{n-1}(E_i; \scrO_{E_i}) = \dim  H^0(E_i; K_{E_i}) = 1$ and $H^{n-1}(E_j; \scrO_{E_j}) =0$ for $j\neq i$. \qed
\end{lemma} 

\begin{corollary}\label{cor1.8}  If $X$ is $0$-liminal, then $H^0(E; \Omega^{n-1}_E/\tau ^{n-1}_E) $ has dimension $0$ or $1$, and has dimension $1$ $\iff$ there exists an $i$ such that $ \dim  H^0(E_i; K_{E_i}) = 1$ and $H^0(E_j; K_{E_j}) =0$ for $j\neq i$.  
\end{corollary}
\begin{proof} This follows from Lemma~\ref{lemma1.8} since $H^0(E; \Omega^{n-1}_E/\tau ^{n-1}_E) \cong \bigoplus_iH^0(E_i; \Omega^{n-1}_{E_i})$ which is dual to $\bigoplus_iH^{n-1}(E_i; \scrO_{E_i})$. 
\end{proof}  

For more discussion, especially in terms of the ``essential divisor" of $\pi\colon \hX \to X$ and the Hodge type of $E$, see \cite{Ishii85} as well as \cite{Fujino2011} and  \cite{Ishii2012}.

We note the following consequence of the Mayer-Vietoris spectral sequence:

\begin{lemma}\label{MVconseq}   Suppose that  $H^i(E;\scrO_E) =0$ for $0< i < m$.  If $\Gamma$ denotes the dual complex of $E$ and $|\Gamma|$ is its topological realization, then $H^i(|\Gamma|; \Cee) =0$ for $0< i < m$.  If  $H^1(E; \scrO_E) =0$,  then  the map
$$\bigoplus_iH^1(E_i; \scrO_{E_i})\to \bigoplus_{\#(I) = 2}H^1(E_I; \scrO_{E_I})$$
is injective. 
\end{lemma}
\begin{proof}  The usual Mayer-Vietoris weight spectral sequence $E_1^{p,q} = \bigoplus_{\#(I) = p+1}H^q(E_I; \Cee) \implies H^{p+q}(E; \Cee) $  degenerates at $E_2$, and hence the Mayer-Vietoris spectral sequence for $\scrO_E$ does as well.  The bottom row of the $E_2$ term computes  $H^i(|\Gamma|; \Cee)$.  Thus, if   $H^i(E; \scrO_E) =0$ for $0< i  < m$, then $H^i(|\Gamma|; \Cee) =0$ for $0< i < n-1$ as well. The argument for the injectivity of $\bigoplus_iH^1(E_i; \scrO_{E_i})\to \bigoplus_{\#(I) = 2}H^1(E_I; \scrO_{E_I})$ in case $H^1(E; \scrO_E) =0$ is similar.
\end{proof}

In Section~\ref{section3},  we will need a lemma due to Namikawa  \cite[Lemma 2.2]{namtop}.  Before stating it,  we first fix some notation.  For any   analytic space $S$, we have the usual map
$$\delta = \frac{d \log}{2\pi\sqrt{-1}} \colon H^1(S; \scrO_S^*) \to H^1(S; \Omega^1_S).$$

\begin{lemma}\label{namlemma}  If $X$ is normal  and   $H^1(\hX; \scrO_{\hX}) =0$, the map
$$H^0(X; R^1\pi_*\scrO_{\hX}^*)\otimes_{\Zee}\Cee  \xrightarrow{\delta\otimes 1} H^0(X; R^1\pi_*\Omega^1_{\hX})$$
is injective. 
\end{lemma}
\begin{proof} Since $X$ is  Stein, $H^q(X; R^p\pi_*\Omega^1_{\hX})=0$ for $q > 0$. Hence, by the Leray spectral sequence, $H^0(X; R^1\pi_*\Omega^1_{\hX} ) = H^1(\hX;\Omega^1_{\hX} )$.  By the exponential sheaf sequence and since $X$ is contractible as well as Stein,  $H^i(X; \scrO_X^*) = 0$ for $i > 0$. Since $X$ is normal, $\pi_*\scrO_{\hX}^* = \scrO_X^*$. Thus, again by the Leray spectral sequence,  $\Pic \hX = H^1(\hX; \scrO_{\hX}^*)  =  H^0(X; R^1\pi_*\scrO_{\hX}^*)$. So it will suffice to prove that $\delta\otimes 1 \colon  \Pic \hX \otimes_{\Zee}\Cee\to H^1(\hX;\Omega^1_{\hX} )$ is injective. 

By a result of Steenbrink \cite[Lemma 2.14]{Steenbrink}, if $H^1(\hX; \scrO_{\hX}) = 0$,  then  $H^1(E; \scrO_E) = 0$ as well. By the exponential sheaf sequence and the fact that $H^1(\hX; \scrO_{\hX}) = H^1(E; \scrO_E) = 0$  there is a commutative diagram
$$\begin{CD}
0 @>>> H^1(\hX; \scrO_{\hX}^*) @>>> H^2(\hX; \Zee) @>>> H^2(\hX; \scrO_{\hX})\\
@. @VVV @VV{\cong}V @VVV \\
 0 @>>> H^1(E; \scrO_E^*) @>>> H^2(E; \Zee) @>>> H^2(E; \scrO_E).
\end{CD}$$
If $X$ is Cohen-Macaulay and Du Bois,  then the map $H^2(\hX; \scrO_{\hX}) \to H^2(E; \scrO_E)$ is an isomorphism, and 
hence $\Pic \hX\cong \Pic E$.  If we just assume that $H^1(\hX; \scrO_{\hX}) =0$, the above diagram still implies that  $\Pic \hX\to \Pic E$ is injective, and hence that $\Pic \hX\otimes_{\Zee}\Cee \to \Pic E\otimes_{\Zee}\Cee$ is injective as well. We have the sequence of homomorphisms 
$$ H^1(\hX;\Omega^1_{\hX} ) \to  H^1(E;\Omega^1_E/\tau^1_E ) \to \bigoplus_i H^1(E_i; \Omega^1_{E_i} ).$$
There is thus a  commutative diagram
$$\begin{CD}
\Pic \hX\otimes_{\Zee}\Cee @>>> \Pic E\otimes_{\Zee}\Cee \\
@V{\delta\otimes 1}VV @VVV\\
H^1(\hX;\Omega^1_{\hX} ) @>>> \bigoplus_i H^1(E_i; \Omega^1_{E_i} )
\end{CD}$$
whose top row is injective. Hence  it suffices to prove that  $\Pic E\otimes_{\Zee}\Cee \to \bigoplus_i H^1(E_i; \Omega^1_{E_i} )$ is injective. 

First we note the following:

\begin{lemma}\label{MVforOstar} There is an exact sequence
$$0 \to H^1(|\Gamma|; \Cee^*) \to \Pic E \to \Ker\Big\{\bigoplus_i\Pic E_i \to  \bigoplus_{\#(I) = 2}\Pic E_I   \Big\}.$$
\end{lemma}
\begin{proof} There is a resolution of $\scrO_E^*$ of the form 
$$\bigoplus_i\scrO_{E_i}^* \to \bigoplus_{\#(I) = 2}\scrO_{E_I}^* \to  \bigoplus_{\#(I) = 3}\scrO_{E_I}^*\to\dots$$
Hence, there is a spectral sequence  $E_1^{p,q} =  \bigoplus_{\#(I) = p+1}H^q(E_I;\scrO_{E_I}^*) \implies H^{p+q}(E;\scrO_E^*)$. (Note however that this spectral sequence does not necessarily degenerate at $E_2$.) There is the natural isomorphism  $E_2^{p,0}\cong H^p(|\Gamma|; \Cee^*)$. The term $E_2^{0,1}$ is just $ \Ker\Big\{\bigoplus_i\Pic E_i \to  \bigoplus_{\#(I) = 2}\Pic E_I   \Big\}$. The   exact sequence
$$0 \to E_2^{1,0} \to H^1(E;\scrO_E^*) \to \Ker d_2|E_2^{0,1}  \to 0$$
then gives the exact sequence of the lemma.
\end{proof}

By Lemma~\ref{MVconseq}, under the assumption that   $H^1(E; \scrO_E) = 0$,   $H^1(|\Gamma|; \Cee) =0$ and hence  $H_1(|\Gamma|; \Zee)$ is a finite abelian group. By the universal coefficient theorem, $$H^1(|\Gamma|; \Cee^*)\cong  \Hom(H_1(|\Gamma|; \Zee), \Cee^*)$$ is also a finite abelian group. 
Hence $H^1(|\Gamma|; \Cee^*) \otimes_{\Zee}\Cee =0$. Then Lemma~\ref{MVforOstar}  implies that the homomorphism $ \Pic E\otimes_{\Zee}\Cee \to  \Big(\Ker\Big\{\bigoplus_i\Pic E_i \to  \bigoplus_{\#(I) = 2}\Pic E_I   \Big\}\Big)\otimes_{\Zee}\Cee $ is injective.

By the second statement in Lemma~\ref{MVconseq}, the differential of the homomorphism  $\psi\colon \bigoplus_i\Pic E_i \to  \bigoplus_{\#(I) = 2}\Pic E_I $ is injective, and hence the restriction of $\psi$ to  $\bigoplus_i\Pic^0 E_i$ has a finite kernel. The kernel of the induced homomorphism 
$$\Ker\Big\{\bigoplus_i\Pic E_i \to  \bigoplus_{\#(I) = 2}\Pic E_I   \Big\} \to \bigoplus_i(\Pic E_i/\Pic^\tau  E_i)$$
is contained in $\bigoplus_i \Pic^\tau  E_i$, where as usual $\Pic^\tau  E_i$ denotes the line bundles on $E_i$ which are numerically equivalent to zero. Since $\Pic^\tau  E_i/\Pic^0  E_i$ is isomorphic to the torsion subgroup of $H^2(E_i;\Zee)$, $\Pic^0  E_i$ has finite index in $ \Pic^\tau  E_i$ for every $i$, and hence  $\bigoplus_i \Pic^0  E_i$ has finite index in $\bigoplus_i \Pic^\tau  E_i$. Thus  the homomorphism $\Ker\Big\{\bigoplus_i\Pic E_i \to  \bigoplus_{\#(I) = 2}\Pic E_I   \Big\} \to \bigoplus_i(\Pic E_i/\Pic^\tau  E_i)$ has a finite kernel. Taking the tensor product with $\Cee$ then gives a sequence of  injections
$$ \Pic E\otimes_{\Zee}\Cee \to  \Big(\Ker\Big\{\bigoplus_i\Pic E_i \to  \bigoplus_{\#(I) = 2}\Pic E_I   \Big\}\Big)\otimes_{\Zee}\Cee   \to \bigoplus_i(\Pic E_i/\Pic^\tau  E_i)\otimes_{\Zee}\Cee.$$

By Hodge theory, $(\Pic E_i/\Pic^\tau  E_i)\otimes_{\Zee}\Cee \to H^1(E_i; \Omega^1_{E_i} )$ is injective for every $i$. Thus $ \Pic E\otimes_{\Zee}\Cee \to  \bigoplus_i H^1(E_i; \Omega^1_{E_i} )$ is injective as well. We have seen that this implies that  $\delta\otimes 1 \colon  \Pic \hX \otimes_{\Zee}\Cee\to H^1(\hX;\Omega^1_{\hX} )$ is injective, completing the proof of Lemma~\ref{namlemma}. 
\end{proof}

\section{Some local deformation theory}\label{section2} 

We keep the notational conventions of Section~\ref{section1}.  Throughout this section, we freely use Schlessinger's isomorphism: if  $\operatorname{depth}_xX \ge 3$, then   $H^0(X; T^1_X) \cong H^1(U; T_U)$ \cite[Lemma 1.17]{FL}. 

\subsection{The main theorem} The goal in this subsection is to analyze  the local deformation theory of $X$  along the lines of  \cite[Theorem 2.1]{FL}.

\begin{theorem}\label{maintheorem}   Let $X$ be a good Stein representative for an isolated Gorenstein $0$-liminal singularity and let $\pi \colon \hX \to X$ be a good equivariant resolution.
\begin{enumerate} 
\item[\rm(i)] There is an exact sequence (arising from the long exact local cohomology sequence)
$$0 \to H^1(\hX; \Omega^{n-1}_{\hX}(\log E)) \to H^1(U; T_U) \to H^2_E( \hX; \Omega^{n-1}_{\hX}(\log E)),$$
and the map $H^1(U; T_U) \to H^2_E( \hX; \Omega^{n-1}_{\hX}(\log E))$ is surjective if $X$ is a local complete intersection or if $H^{n-3}(E; \Omega^1_E/\tau^1_E) = 0$. 
Moreover, in case $X$ is a local complete intersection, 
$$\dim  H^2_E( \hX; \Omega^{n-1}_{\hX}(\log E)) = b^{1,n-2} = s_1.$$
\item[\rm(ii)] There is an exact sequence
 \begin{gather*}
  H^0(E; \Omega^{n-1}_{\hX}(\log E)|E)\to H^1(\hX; \Omega^{n-1}_{\hX}(\log E)(-E)) 
\to H^1(\hX; \Omega^{n-1}_{\hX}(\log E))\to  \\ \to H^1(E; \Omega^{n-1}_{\hX}(\log E)|E)\to 0. 
\end{gather*}
 Moreover, $H^0(E; \Omega^{n-1}_{\hX}(\log E)|E) \cong H^0(E; \Omega^{n-1}_E/\tau ^{n-1}_E)$.  Thus, by Corollary~\ref{cor1.8},  the dimension $\ell^{n-1, 0}$ of $H^0(E; \Omega^{n-1}_{\hX}(\log E)|E) $ is either  $0$ or $1$ and $H^0(E; \Omega^{n-1}_{\hX}(\log E)|E)  = 0$  $\iff$ $H^{n-1}(E_i; \scrO_{E_i})=0$ for every $i$. 
\item[\rm(iii)]    $H^1(\hX; \Omega^{n-1}_{\hX})$  and $H^1(\hX; \Omega^{n-1}_{\hX}(\log E)(-E))$ have the same image in $H^1(U; T_U)$.  More precisely, if $\im H^1(\hX; \Omega^{n-1}_{\hX}(\log E)(-E))$ denotes the image of $H^1(\hX; \Omega^{n-1}_{\hX}(\log E)(-E))$ in $H^1(\hX; \Omega^{n-1}_{\hX})$, there is a canonical isomorphism
$$H^1(\hX; \Omega^{n-1}_{\hX}) \cong \im H^1(\hX; \Omega^{n-1}_{\hX}(\log E)(-E)) \oplus H^1_E(\hX; \Omega^{n-1}_{\hX})$$
such that the map $H^1(\hX; \Omega^{n-1}_{\hX}) \to H^1(U; T_U)$ is the natural map on  $H^1(\hX; \Omega^{n-1}_{\hX}(\log E)(-E))$ and is $0$ on the $H^1_E(\hX; \Omega^{n-1}_{\hX})$ summand. Moreover, the map $H^0(\hX;  \Omega^{n-1}_{\hX}) \to H^0(U;  \Omega^{n-1}_{\hX}|U)$ is an isomorphism.
\item[\rm(iv)] With assumptions as in {\rm(i)},   let $K = \Ker\{H^2_E(\hX; \Omega^{n-1}_{\hX}) \to H^2(\hX; \Omega^{n-1}_{\hX})\}$. Then 
$$K\cong H^2_E(\hX; \Omega^{n-1}_{\hX}(\log E)(-E)),$$ and  there are exact sequences
\begin{gather*}
  H^1(\hX; \Omega^{n-1}_{\hX}(\log E)(-E)) \to H^1(U; T_U) \to K \to 0;\\
0 \to \Gr^{n-1}_FH^n (L) \to K \to H^2_E( \hX; \Omega^{n-1}_{\hX}(\log E)) \to 0.
\end{gather*}
\item[\rm(v)] $\dim   H^0_E(\hX; \Omega^n_{\hX}) = 1$ and the natural map  $H^1_E(\hX; \Omega^n_{\hX}) \to  H^{n+1}_E(\hX)$ is injective. Its  image $\im H^1_E(\hX; \Omega^n_{\hX})$ is equal to $F^nH^{n+1}_E(\hX) \cong \Gr^n_FH^{n+1}_E(\hX)$.  There is an induced map $H^2_E(\hX; \Omega^{n-1}_{\hX}) \to H^{n+1}_E(\hX)/\im H^1_E(\hX; \Omega^n_{\hX})$. 
\item[\rm(vi)]  The image of  $H^2_E(\hX; \Omega^{n-1}_{\hX}) \to H^{n+1}_E(\hX)/\im H^1_E(\hX; \Omega^n_{\hX})$ contains   $\Gr^{n-1}_FH^{n+1}_E(\hX)$.  Let $K' =\Ker\{H^2_E(\hX; \Omega^{n-1}_{\hX}) \to H^{n+1}_E(\hX)/\im H^1_E(\hX; \Omega^n_{\hX})\}$.  Then $K' \subseteq K$. With assumptions as in {\rm(i)},  there is an inclusion 
$$K'\oplus \Gr^{n-1}_FH^n (L) \hookrightarrow K.$$
Hence the induced map $K' \to H^2_E( \hX; \Omega^{n-1}_{\hX}(\log E))$ is injective.
Thus $K'\oplus \Gr^{n-1}_FH^n (L) \cong  K$  $\iff$ $\dim K' = b^{1, n-2}$  $\iff$ the   map $K'\to H^2_E( \hX; \Omega^{n-1}_{\hX}(\log E))$ induced by the surjection $K \to H^2_E( \hX; \Omega^{n-1}_{\hX}(\log E))$ is an isomorphism. Finally, if the image of  $H^2_E(\hX; \Omega^{n-1}_{\hX}) \to H^{n+1}_E(\hX)/\im H^1_E(\hX; \Omega^n_{\hX})$ is equal to  $\Gr^{n-1}_FH^{n+1}_E(\hX)$, then  $$K'\oplus \Gr^{n-1}_FH^n (L) \cong K$$ and hence $K'\cong H^2_E( \hX; \Omega^{n-1}_{\hX}(\log E))$.
\item[\rm(vii)]  If $X$ is a local complete intersection, then the group $H^2_E( \hX; \Omega^{n-1}_{\hX}(\log E))$ is not $0$. Hence, if the map $K' \to H^2_E( \hX; \Omega^{n-1}_{\hX}(\log E))$ is an isomorphism, then $K'\neq 0$. 
\end{enumerate} 
\end{theorem} 

 \begin{remark}  In \cite{FL}, the key idea is to compare the tangent space $H^0(X; T^1_X)= H^1(U; T_U)$ to deformations of $(X,x)$ to the image of the tangent space to deformations of $\hX$, namely $H^1(\hX; T_{\hX})$.  In the crepant case, $H^1(\hX; T_{\hX})$ is the same as $H^1(\hX; \Omega^{n-1}_{\hX})$. In general $H^1(\hX; T_{\hX})$ and $H^1(\hX; \Omega^{n-1}_{\hX})$ are not independent of the choice of a resolution and must be replaced by birationally invariant spaces such as $H^1(\hX; \Omega^{n-1}_{\hX}(\log E)(-E))$ and $H^1(\hX; \Omega^{n-1}_{\hX}(\log E))$. These spaces are closely connected to the Du Bois and link invariants of $X$ of Definition \ref{defineinv}, and in particular the $1$-Du Bois and $1$-rational conditions of Definition~\ref{defkDB}   come into play.    Passing from rational  to $0$-liminal (i.e.\ Du Bois but not rational) singularities leads to the following differences between Theorem~\ref{maintheorem} above and \cite[Theorem 2.1]{FL}:
\begin{enumerate}
\item[(i)]  In Theorem~\ref{maintheorem}(ii), we do not claim that the map   
$$H^0(E; \Omega^{n-1}_{\hX}(\log E)|E)\to H^1(\hX; \Omega^{n-1}_{\hX}(\log E)(-E)) $$ is injective. However, since $\dim H^0(E; \Omega^{n-1}_{\hX}(\log E)|E)$ is either $0$ or $1$, it is always close to being injective. 
\item[(ii)]  In Theorem~\ref{maintheorem}(v) and (vi), we have to replace $H^{n+1}_E(\hX)$ by $H^{n+1}_E(\hX)/\im H^1_E(\hX; \Omega^n_{\hX})$, where $\dim H^1_E(\hX; \Omega^n_{\hX}) =1$. 
\item[(iii)] The statement in Theorem~\ref{maintheorem}(vii) that $H^2_E( \hX; \Omega^{n-1}_{\hX}(\log E)) \neq 0$ plays the role in this setting of the hypothesis in \cite{FL} that $X$ is $1$-irrational, see Proposition~\ref{prop1.12} below. Of course, in our situation, $X$ is already non-rational, i.e.\ $0$-irrational. 
\end{enumerate}
\end{remark}

\begin{proof}[Proof of Theorem~\ref{maintheorem}]  (i) Identifying $ T_U$ with $\Omega^{n-1}_{\hX}(\log E)|U$, the local cohomology sequence gives an exact sequence
$$H^1_E(\hX; \Omega^{n-1}_{\hX}(\log E)) \to H^1(\hX; \Omega^{n-1}_{\hX}(\log E)) \to H^1(U; T_U) \to H^2_E( \hX; \Omega^{n-1}_{\hX}(\log E)) $$ where the cokernel of the right-hand map is contained in $H^2(\hX; \Omega^{n-1}_{\hX}(\log E))$. 
 By  Lemma~\ref{extravan}, since $X$ is Du Bois, $H^1_E(\hX; \Omega^{n-1}_{\hX}(\log E))  =0$ and   hence the map $H^1(\hX; \Omega^{n-1}_{\hX}(\log E)) \to H^1(U; T_U)$ is injective. By  Proposition~\ref{locinvvan}(i), $H^2(\hX; \Omega^{n-1}_{\hX}(\log E)(-E)) = H^3(\hX; \Omega^{n-1}_{\hX}(\log E)(-E))  =0$ and hence 
$$H^2(\hX; \Omega^{n-1}_{\hX}(\log E)) \cong H^2(E; \Omega^{n-1}_{\hX}(\log E)|E)= \Gr^{n-1}_FH^{n+1}(L).$$   If   $X$ is a local complete intersection, then $\ell^{n-1, 2}=\dim  \Gr^{n-1}_FH^{n+1}(L)=0$ by  Proposition~\ref{locinvvan}(iv),  and therefore  $H^2(\hX; \Omega_{\hX}^{n-1}(\log E)) = 0$.    More generally, 
$$\dim \Gr^{n-1}_FH^{n+1}(L) = \ell^{n-1, 2} = \ell^{1, n-3} = \dim \Gr^1_FH^{n-2}(L)=\dim H^{n-3}(E; \Omega^1_{\hX}(\log E)|E).$$
By Proposition~\ref{locinvvan}(iii), $H^{n-2}(E) \to H^{n-2}(L)$ is surjective, and thus 
$$\Gr^1_FH^{n-2}(E) =H^{n-3}(E; \Omega^1_E/\tau^1_E) \to \Gr^1_FH^{n-2}(L)$$ is surjective. Thus, if either $X$ is a local complete intersection or  $H^{n-3}(E; \Omega^1_E/\tau^1_E) = 0$, then $H^2(\hX; \Omega^{n-1}_{\hX}(\log E)) =0$.  

To see the final statement, note that $H^2_E(\hX; \Omega^{n-1}_{\hX}(\log E))$ is dual to $H^{n-2}(\hX; \Omega^1_{\hX}(\log E)(-E))$ and thus has dimension $b^{1, n-2}$ by definition. By the proof of \cite[Theorem 4.5]{FL22d}, if $X$ is a local complete intersection, since $b^{1,n-1} =0$, 
$$-b^{1,n-2} = b^{1,n-1} - b^{1,n-2} = \ell^{0,n} - \ell^{0,n-1} - (s_1 - s_n).$$
Also, $\ell^{0,n}=0$ for dimension reasons, and for every Du Bois singularity (not necessarily $0$-liminal) 
$$\ell^{0,n-1} = \dim H^{n-1}(E; \scrO_E) = \dim H^{n-1}(\hX; \scrO_{\hX}) = s_n,$$
by \cite[Proposition 2.13]{Steenbrink}. Thus $b^{1,n-2} = s_1$ for an arbitrary isolated local complete intersection Du Bois singularity. (Compare \cite[Theorem 6]{SteenbrinkDB} for the case $n=3$.)

\smallskip
\noindent (ii) We have the link exact sequence 
\begin{gather*}
H^0(E; \Omega^{n-1}_{\hX}(\log E)|E) \to  H^1(\hX; \Omega^{n-1}_{\hX}(\log E)(-E)) \to H^1(\hX; \Omega^{n-1}_{\hX}(\log E)) \\
\to H^1(E; \Omega^{n-1}_{\hX}(\log E)|E) 
\to H^2(\hX; \Omega^{n-1}_{\hX}(\log E)(-E)),
\end{gather*}
 with $H^2(\hX; \Omega^{n-1}_{\hX}(\log E)(-E)) =0$ by Proposition~\ref{locinvvan}(i).
 
 To prove the remaining statements, note that,  by Proposition~\ref{locinvvan}(iii), there is  an exact sequence
 $$0 \to H^{n-1}_E(\hX) \to H^{n-1}(E) \to H^{n-1}(L) \to 0,$$ and hence an exact sequence on the graded pieces of the Hodge filtration
 $$0 \to \Gr_F^{n-1}H^{n-1}_E(\hX) \to H^0(E; \Omega^{n-1}_E/\tau ^{n-1}_E) \to H^0(E; \Omega^{n-1}_{\hX}(\log E)|E)  \to 0.$$
 But   the map 
 $H^0(E; \Omega^{n-1}_E/\tau ^{n-1}_E) \to H^0(E; \Omega^{n-1}_{\hX}(\log E)|E)$ is injective as  the map on sheaves $\Omega^{n-1}_E/\tau ^{n-1}_E \to \Omega^{n-1}_{\hX}(\log E)|E$ is injective by Proposition~\ref{linkexseq}, or directly because $\Gr_F^{n-1}H^{n-1}_E(\hX) = H^{-1}(E; \Omega^{n-1}_{\hX}(\log E)/\Omega^{n-1}_{\hX}) =0$. Hence 
 $$H^0(E; \Omega^{n-1}_E/\tau ^{n-1}_E) \cong H^0(E; \Omega^{n-1}_{\hX}(\log E)|E)$$ and $H^0(E; \Omega^{n-1}_{\hX}(\log E)|E) =0$ $\iff$ $H^0(E; \Omega^{n-1}_E/\tau ^{n-1}_E) =0$. The remaining statements follow from Corollary ~\ref{cor1.8}. 
 
 \smallskip
\noindent (iii) There is a commutative diagram
$$\begin{CD}
   @. H^1_E(\hX; \Omega^{n-1}_{\hX}) @.\\
    @. @VVV @.  \\
   H^1(\hX; \Omega^{n-1}_{\hX}(\log E)(-E)) @>>> H^1(\hX; \Omega^{n-1}_{\hX}) @>>> H^1(E; \Omega^{n-1}_E/\tau^{n-1}_E) @>>>  0
\\
 @VVV @VVV @.\\
  H^1(U; T_U) @>{=}>> H^1(U; T_U). @. 
\end{CD}$$
Lemma~\ref{extravan} implies that  $H^1_E(\hX; \Omega^{n-1}_{\hX}(\log E))=0$, and hence that
$$
H^1_E(\hX; \Omega^{n-1}_{\hX})   \cong H^0_E(\hX; \Omega^{n-1}_{\hX}(\log E)/ \Omega^{n-1}_{\hX}) = H^0(\hX; \Omega^{n-1}_{\hX}(\log E)/ \Omega^{n-1}_{\hX}) = \Gr^{n-1}_FH^n_E(\hX) .$$
Then Proposition~\ref{locinvvan}(iii) gives $H^n_E(\hX) \cong H^n(E)$, and hence 
$$\Gr^{n-1}_FH^n_E(\hX)\cong \Gr^{n-1}_FH^n(E) = H^1(E; \Omega^{n-1}_E/\tau^{n-1}_E).$$
Moreover the isomorphism $H^1_E(\hX; \Omega^{n-1}_{\hX}) \cong H^1(E; \Omega^{n-1}_E/\tau^{n-1}_E)$ is induced by the sequence of maps 
$$H^1_E(\hX; \Omega^{n-1}_{\hX}) \to  H^1(\hX; \Omega^{n-1}_{\hX}) \to H^1(E; \Omega^{n-1}_E/\tau^{n-1}_E).$$ Thus the map $H^1_E(\hX; \Omega^{n-1}_{\hX}) \to H^1(\hX; \Omega^{n-1}_{\hX})$ is injective and its image splits the exact sequence
$$0 \to \im H^1(\hX; \Omega^{n-1}_{\hX}(\log E)(-E))  \to H^1(\hX; \Omega^{n-1}_{\hX}) \to  H^1(E; \Omega^{n-1}_E/\tau^{n-1}_E) \to   0.$$
This proves the first two statements in (iii), and the final sentence follows from the local cohomology sequence
$$0 \to H^0_E(\hX;  \Omega^{n-1}_{\hX}) \to H^0(\hX;  \Omega^{n-1}_{\hX}) \to H^0(U;  \Omega^{n-1}_{\hX}|U)\to H^1_E(\hX; \Omega^{n-1}_{\hX}) \to H^1(\hX; \Omega^{n-1}_{\hX}),$$
the injectivity of the map  $H^1_E(\hX; \Omega^{n-1}_{\hX}) \to H^1(\hX; \Omega^{n-1}_{\hX})$,
and the fact that $H^0_E(\hX;  \Omega^{n-1}_{\hX}) =0$. 

\smallskip
\noindent (iv) The proof is essentially identical to that of \cite[Theorem 2.1(v)]{FL}. 

\smallskip
\noindent (v) $ H^1_E(\hX; \Omega^n_{\hX})$ is dual to $H^{n-1}(\hX; \scrO_{\hX})$, hence has dimension $1$ by Theorem~\ref{Thm1.1}.  The map $H^1_E(\hX; \Omega^n_{\hX})\to H^{n+1}_E(\hX)$ is induced by the long exact sequence of local hypercohomology associated to
\begin{equation}\label{1.00}
0 \to \Omega^n_{\hX}[n] \to \Omega^\bullet_{\hX} \to \Omega^\bullet_{\hX}/\Omega^n_{\hX}[n] \to 0.
\end{equation} 
From the exact sequence
$$0 = H^0_E(\hX;  \Omega^n_{\hX}(\log E)) \to H^0_E(\hX;  \Omega^n_{\hX}(\log E)/\Omega^n_{\hX}) \to H^1_E(\hX; \Omega^n_{\hX}) \to H^1_E(\hX;  \Omega^n_{\hX}(\log E))$$
and the fact that $H^1_E(\hX;  \Omega^n_{\hX}(\log E))$ is dual to $H^{n-1}(\hX; \scrO_{\hX}(-E)) = 0$ by the Du Bois condition, 
$$ H^0_E(\hX;  \Omega^n_{\hX}(\log E)/\Omega^n_{\hX}) \cong H^1_E(\hX; \Omega^n_{\hX}).$$
By Hodge theory (Proposition~\ref{linkexseq}), there is a  spectral sequence with $E_1$ term 
$$E_1^{p,q} = H^q(\hX; \Omega^p_{\hX}(\log E)/\Omega^p_{\hX}) \implies \mathbb{H}^{p+q}(\hX; \Omega^\bullet_{\hX}(\log E)/\Omega^\bullet_{\hX})= H^{p+q+1}_E(\hX),$$
which degenerates at $E_1$, and the corresponding filtration on $H^{p+q+1}_E(\hX)$ is the Hodge filtration. In particular,  the map $H^0(\hX;  \Omega^n_{\hX}(\log E)/\Omega^n_{\hX})  =  H^0_E(\hX;  \Omega^n_{\hX}(\log E)/\Omega^n_{\hX})  \to H^{n+1}_E(\hX)$ is injective, and its image is $\Gr^n_FH^{n+1}_E(\hX)$. Then as in the proof of \cite[Theorem 2.1(vi)]{FL}, the isomorphism $$\mathbb{H}^{k-1}(\Omega^\bullet_{\hX}(\log E)/\Omega^\bullet_{\hX}) = \mathbb{H}^{k-1} _E(\Omega^\bullet_{\hX}(\log E)/\Omega^\bullet_{\hX}) \cong \mathbb{H}^k_E(\Omega^\bullet_{\hX})\cong H^k_E(\hX)$$
 leads to a commutative diagram 
 $$\begin{CD}
H^0(\hX; \Omega^n_{\hX}(\log E)/\Omega^n_{\hX}) @>{\cong}>> H^1_E(\hX; \Omega^n_{\hX}) \\
@VVV @VVV\\
H^{n+1}_E(\hX) @>{=}>> H^{n+1}_E(\hX).
\end{CD}$$ 
Hence the map $H^1_E(\hX; \Omega^n_{\hX})\to H^{n+1}_E(\hX)$ is also injective with image $F^nH^{n+1}_E(\hX) \cong \Gr^n_FH^{n+1}_E(\hX)$, as  $F^{n+1}H^{n+1}_E(\hX) =0$. 

The exact sequence (\ref{1.00}) leads to an exact sequence
$$0 \to H^1_E(\hX; \Omega^n_{\hX})\to H^{n+1}_E(\hX) \to \mathbb{H}^{n+1}_E(\hX; \Omega^\bullet_{\hX}/\Omega^n_{\hX}[n]) \to H^2_E(\hX; \Omega^n_{\hX}).$$
But $H^2_E(\hX; \Omega^n_{\hX})$ is dual to $H^{n-2}(\hX; \scrO_{\hX}) = 0$,   by Lemma~\ref{lemma1.1}. Thus, 
$$H^{n+1}_E(\hX)/\im H^1_E(\hX; \Omega^n_{\hX}) \cong \mathbb{H}^{n+1}_E(\hX; \Omega^\bullet_{\hX}/\Omega^n_{\hX}[n]).$$
The map $H^2_E(\hX; \Omega^{n-1}_{\hX}) \to H^{n+1}_E(\hX)/\im H^1_E(\hX; \Omega^n_{\hX})$ is then just the edge homomorphism.

\smallskip
\noindent (vi) By the Du Bois condition, $H^i_E(\Omega^n_{\hX}(\log E)) = 0$ for $i< n$ since $H^i_E(\Omega^n_{\hX}(\log E))$ is dual to $H^{n-i}(\hX; \scrO_{\hX}(-E))=0$. Hence, the local hypercohomology $\mathbb{H}^i_E$ of the truncated complex $ \Omega^\bullet_{\hX}(\log E)/  \Omega^n_{\hX}(\log E)[n]$ is $0$ for $i < 2n$ and in particular for $i=n, n+1$. Thus, working with the truncated complexes $\Big(\Omega^\bullet_{\hX}(\log E)/\Omega^\bullet_{\hX}\Big)\Big/\Big (\Omega^n_{\hX}(\log E)/\Omega^n_{\hX}\Big)[n]$ and $\Omega^\bullet_{\hX}/\Omega^n_{\hX}[n]$ instead of the full complexes, the long exact hypercohomology sequence shows that
$$\mathbb{H}^n_E\Big(\Big(\Omega^\bullet_{\hX}(\log E)/\Omega^\bullet_{\hX}\Big)\Big/\Big (\Omega^n_{\hX}(\log E)/\Omega^n_{\hX}\Big)[n]\Big) \cong \mathbb{H}^{n+1}_E\Big(\Omega^\bullet_{\hX}/\Omega^n_{\hX}[n]\Big).$$
A straightforward variation of the argument  for the   proof of \cite[Theorem 2.1(vi)]{FL}, shows that  the image of  
$$H^2_E(\hX; \Omega^{n-1}_{\hX}) \to H^{n+1}_E(\hX)/\im H^1_E(\hX; \Omega^n_{\hX})$$ contains $\Gr^{n-1}_FH^{n+1}_E(\hX)$.   More precisely, the argument there shows that the diagram
$$\begin{CD}
H^1(\hX; \Omega^{n-1}_{\hX}(\log E)/\Omega^{n-1}_{\hX}) @>>> H^2_E(\hX; \Omega^{n-1}_{\hX}) \\
@VVV @VV{\beta}V\\
H^{n+1}_E(\hX)/\Gr^n_FH^{n+1}_E(\hX) @>{=}>> H^{n+1}_E(\hX)/\Gr^n_FH^{n+1}_E(\hX)
\end{CD}$$ 
is commutative and that the  inclusion homomorphism $$\Gr_F^{n-1}H^{n+1}_E(\hX) = H^1_E(\hX; \Omega^{n-1}_{\hX}(\log E)/\Omega^{n-1}_{\hX}) \to \im \beta= \im H^2_E(\hX; \Omega^{n-1}_{\hX})$$ is given by the coboundary 
$$\partial\colon H^1_E(\hX; \Omega^{n-1}_{\hX}(\log E)/\Omega^{n-1}_{\hX}) \to H^2_E(\hX; \Omega^{n-1}_{\hX}) $$ followed by the surjection $ H^2_E(\hX; \Omega^{n-1}_{\hX}) \to \im  \beta$.

Next, taking the usual hypercohomology of the 
exact sequence (\ref{1.00}) and using the Grauert-Riemenschneider theorem that $H^i(\hX; K_{\hX}) = 0$ for $i>0$ gives: for $i>n$, 
$$ H^i(\hX) = \mathbb{H}^i(\hX; \Omega^\bullet_{\hX}) \cong \mathbb{H}^i(\hX; \Omega^\bullet_{\hX}/\Omega^n_{\hX}[n]).$$
In particular, there is an induced map $H^2(\hX;  \Omega^{n-1}_{\hX}) \to H^{n+1}(\hX)$. Via the natural map  $H^{n+1}_E(\hX) \to H^{n+1}(\hX)$, the image of $ H^1_E(\hX; \Omega^n_{\hX})$ is sent to the image of  $H^1(\hX; \Omega^n_{\hX}) =0$. Thus there is a commutative diagram:
$$\begin{CD}
H^2_E(\hX; \Omega^{n-1}_{\hX}) @>>> H^2(\hX; \Omega^{n-1}_{\hX}) \\
@VVV @VVV\\
H^{n+1}_E(\hX)/ \im H^1_E(\hX; \Omega^n_{\hX}) @> >> H^{n+1}(\hX).
\end{CD}$$
The argument of \cite[Lemma 2.3(i)]{FL} then shows that the right hand vertical map in the above diagram is injective, and hence that
\begin{align*}
K' &=\Ker\{H^2_E(\hX; \Omega^{n-1}_{\hX}) \to H^{n+1}_E(\hX) /\im H^1_E(\hX; \Omega^n_{\hX})\}\\
&=\Ker\{H^2_E(\hX; \Omega^{n-1}_{\hX}) \to H^{n+1}_E(\hX) /\Gr^n_F H^{n+1}_E(\hX)\} 
\end{align*}
  is contained in  $\Ker\{ H^2_E(\hX; \Omega^{n-1}_{\hX}) \to H^2(\hX; \Omega^{n-1}_{\hX})\} = K$. 

Next we claim that $K'\cap \Gr_F^{n-1}H^n(L) =  0$. There is a  commutative diagram 
$$\begin{CD}
\Gr_F^{n-1}H^n(L) = H^1(\Omega^{n-1}_{\hX}(\log E)|E) @>>> K \\
@VVV @VVV\\
H^n(L)/\Gr^n_FH^n(L) @>>> H^{n+1}_E(\hX)/\Gr^n_F H^{n+1}_E(\hX).
\end{CD}$$
By the $E_1$ degeneration of the Hodge spectral sequence for $H^n(L)=\mathbb{H}^n(E; \Omega^{\bullet}_{\hX}(\log E)|E)$, the left vertical arrow is injective. By Proposition~\ref{locinvvan}(iii)  and strictness, the lower horizontal map  is injective. Hence  the composed map $\Gr_F^{n-1}H^n(L) \to  H^{n+1}_E(\hX)/\Gr^n_F H^{n+1}_E(\hX)$ is injective. It follows that if   $\alpha \in K'\cap \Gr_F^{n-1}H^n(L)$, then $\alpha$ maps to $0$ in $H^{n+1}_E(\hX)/\im H^1_E(\hX; \Omega^n_{\hX})= H^{n+1}_E(\hX)/\Gr^n_F H^{n+1}_E(\hX)$ by the definition of $K'$, and therefore $\alpha =0$. In particular, the map $K'\oplus \Gr^{n-1}_FH^n (L) \to K$ is an inclusion, and the induced map $K' \to H^2_E( \hX; \Omega^{n-1}_{\hX}(\log E))$ is injective. Clearly $K'\oplus \Gr^{n-1}_FH^n (L) \to K$ is an isomorphism $\iff$ the  induced  map $K'\to H^2_E( \hX; \Omega^{n-1}_{\hX}(\log E))$ is an isomorphism  $\iff$ $\dim K' = \dim H^2_E( \hX; \Omega^{n-1}_{\hX}(\log E)) = b^{1, n-2}$.

Finally,  suppose that  the image of $\beta \colon H^2_E(\hX; \Omega^{n-1}_{\hX})\to H^{n+1}_E(\hX)/\im H^1_E(\hX; \Omega^n_{\hX})$ is equal to $\Gr_F^{n-1}H^{n+1}_E(\hX)$. Then  the discussion at the beginning of the proof of (vi) shows that
\begin{align*}
\Gr_F^{n-1}H^{n+1}_E(\hX) &= H^1(\hX; \Omega^{n-1}_{\hX}(\log E)/\Omega^{n-1}_{\hX}) = H^1_E(\hX; \Omega^{n-1}_{\hX}(\log E)/\Omega^{n-1}_{\hX}) \\
&\xrightarrow{\partial}  H^2_E(\hX; \Omega^{n-1}_{\hX}) \xrightarrow{\beta}  \im \beta = \Gr_F^{n-1}H^{n+1}_E(\hX)
\end{align*}
is the identity, so that $\partial$
splits the surjection $H^2_E(\hX; \Omega^{n-1}_{\hX}) \to \Gr_F^{n-1}H^{n+1}_E(\hX)$. Thus  
$$H^2_E(\hX; \Omega^{n-1}_{\hX}) = K'\oplus  \Gr_F^{n-1}H^{n+1}_E(\hX).$$ Then   $K = K'\oplus K''$, where $K''$ is the kernel of 
$$\Gr_F^{n-1}H^{n+1}_E(\hX) \to H^2(\hX; \Omega^{n-1}_{\hX}) \cong H^2(E; \Omega^{n-1}_E/\tau^{n-1}_E) = \Gr_F^{n-1}H^{n+1}(E).$$
Equivalently,  $K''$ is the kernel of $\Gr_F^{n-1}H^{n+1}_E(\hX) \to \Gr_F^{n-1}H^{n+1}(E)$
which by Proposition~\ref{locinvvan}(iii)  is exactly $\Gr_F^{n-1}H^n(L) = H^1(\hX;\Omega^{n-1}_{\hX}(\log E)|E)$.  Thus, under the assumption that  
$$\im \left\{H^2_E(\hX; \Omega^{n-1}_{\hX})\to   H^{n+1}_E(\hX)/\im H^1_E(\hX; \Omega^n_{\hX}) \right\}=\Gr_F^{n-1}H^{n+1}_E(\hX),$$  
we have $K \cong \Gr_F^{n-1}H^n(L) \oplus K'$. It follows that    the surjection $K \to H^2_E( \hX; \Omega^{n-1}_{\hX}(\log E))$, whose kernel is $\Gr_F^{n-1}H^n(L)$,   identifies $K'$ with $H^2_E( \hX; \Omega^{n-1}_{\hX}(\log E))$. 

 \smallskip
\noindent (vii) 
Note that $H^2_E( \hX; \Omega^{n-1}_{\hX}(\log E))$ is dual to $H^{n-2}(\hX; \Omega^1_{\hX}(\log E)(-E))$.   Assume by contradiction that $H^2_E( \hX; \Omega^{n-1}_{\hX}(\log E))=0$. Then by duality $b^{1,n-2} =0$. Since $X$ is $0$-Du Bois and   $X$ is a local complete intersection by assumption,    $X$ is $1$-Du Bois   by Proposition ~\ref{charmDB}. In this case, $X$ is rational by \cite[Corollary 5.5]{FL22d}. But this contradicts the $0$-liminal assumption. 
\end{proof}

  \subsection{Comments} In this subsection, we will make some comments on Theorem~\ref{maintheorem}. First we give some numerical consequences along the lines of those in \cite[Theorem 2.1(ii),(iv)]{FL} and using \cite[Corollary 1.8]{FL}. Since the arguments are very similar to those of  \cite{FL}, we omit the proof. 
  
  \begin{proposition} Under the assumptions of Theorem~\ref{maintheorem}, suppose either that $X$ is a local complete intersection or that $H^{n-3}(E; \Omega^1_E/\tau^1_E) = 0$. Then
   \begin{align*}
   \dim H^0(X; T^1_X) &= b^{1, n-2} + b^{n-1, 1} + \ell^{n-1, 1} - \ell^{n-1,0}\\
   &= b^{2, n-2} + a+ b^{n-1, 1} + \ell^{n-1, 1} - \ell^{n-1,0},
   \end{align*}
   where $a = \dim \Ker\{d: H^3_E(\hX; \Omega^{n-3}_{\hX}(\log E)) \to H^3_E(\hX; \Omega^{n-2}_{\hX}(\log E))\} $ and $\ell^{n-1,0} =0$ or $1$. \qed
   \end{proposition}

  Next, for future use, we elaborate on the meaning of the conditions of Theorem~\ref{maintheorem}(vi), (vii):
  
  \begin{proposition}\label{prop1.12} Let $X$ be  a $0$-liminal hypersurface singularity. Suppose that the map $K' \to  H^2_E( \hX; \Omega^{n-1}_{\hX}(\log E))$ is an isomorphism. Then 
 the kernel of the map $H^0(X; T^1_X) \to K'$  is  contained in  $\mathfrak{m}_xH^0(X; T^1_X)$. 
  \end{proposition}  
  \begin{proof}   First, by Theorem~\ref{maintheorem}(vii), $K'\neq 0$. 
  By Theorem~\ref{maintheorem}(i), there is a commutative diagram whose top row is exact:   
  $$\begin{CD} 
  0  @>>>  H^1( \hX; \Omega^{n-1}_{\hX}(\log E))   @>>>   H^1(U; T_U)  @>>>   H^2_E( \hX; \Omega^{n-1}_{\hX}(\log E))  @>>>  0\\
  @. @. @| @| @. \\
  @. @.    H^0(X; T^1_X)   @>>>   K' . @.
  \end{CD}$$
  Thus   the kernel of $H^0(X; T^1_X) \to K'$  is identified with $H^1( \hX; \Omega^{n-1}_{\hX}(\log E))$. 
  Assume  by contradiction that the image of $H^1( \hX; \Omega^{n-1}_{\hX}(\log E)) $ in $ H^0(X; T^1_X)$  is  not contained in  $\mathfrak{m}_xH^0(X; T^1_X)$.  Then it contains a generator of $H^0(X; T^1_X)$.  Hence $H^1( \hX; \Omega^{n-1}_{\hX}(\log E)) =H^0(X; T^1_X)$,  contradicting  the fact that $K'\neq 0$. 
  \end{proof}

 \subsection{Deformations of the resolution}  Given the  resolution $\pi\colon \hX \to X$, there is the natural map $H^1(\hX; T_{\hX}) \to H^1(U, T_U)$. Since $n =\dim X \ge 3$, $H^1(\hX; \scrO_{\hX}) = R^1\pi_*\scrO_{\hX}=0$, and thus by Wahl's theory \cite[Theorem 1.4(c)]{Wahl} there is a corresponding map of deformation functors $\mathbf{Def}_{\hX} \to \mathbf{Def}_X$. Note that the hypothesis $\dim X \ge 3$ is essential here. For example, if $X$ is   a smoothable simple elliptic singularity  and $\hX$ is the minimal resolution, then the exceptional divisor $E$ will disappear in a general small deformation $\widehat{\mathcal{X}}\to \Delta$ of $\hX$. If   there were a   morphism $\widehat{\mathcal{X}} \to \mathcal{X}$ for some deformation $\mathcal{X}$ of $X$, then $\widehat{\mathcal{X}}$   would be a small resolution of the Cohen-Macaulay singularity on $\mathcal{X}$ corresponding to the singular point of  $X$.  But  a Leray spectral sequence argument shows that $H^1(E; \scrO_E) =0$ (cf.\ \cite[Proposition 1]{Pinkham81}), a contradiction. 
  
  In general, the space $H^1(\hX; T_{\hX})$ and its image in $H^1(U, T_U)$ need not have birational meaning (i.e.\ need not be independent of the resolution). The situation is better in case there is a good anticanonical resolution:
  
  \begin{proposition}\label{Prop2.3} Suppose that $\pi\colon \hX \to X$ is a good anticanonical resolution.
  \begin{enumerate}
  \item[\rm(i)] $T_{\hX}(-E) \cong \Omega^{n-1}_{\hX}$ and $T_{\hX}(-\log E) \cong \Omega^{n-1}_{\hX}(\log E)$.
  \item[\rm(ii)] The natural map $H^1(\hX; T_{\hX}(-\log E)) \to H^1(\hX; T_{\hX})$ is injective, and it is an isomorphism if $H^1(E_i; N_{E_i/\hX}) =0$ for every $i$,  i.e.\ if the $E_i$ are stable submanifolds.
  \item[\rm(iii)]  If $H^1(E; \scrO_E(E)) =0$, then every first order deformation of $\hX$ preserves the divisor $E$ (but not necessarily the divisors $E_i$). 
  \item[\rm(iv)] We have
  $$H^2(\hX;T_{\hX}(-\log E))  \cong \Gr^{n-1}_FH^{n+1}(L).$$ Suppose either that  $X$ is a local complete intersection or that $E$ is smooth and $h^{n-2,2}(E)  =0$.  Then $H^2(\hX;T_{\hX}(-\log E)) = 0$, i.e.\ the functor of deformations of $\hX$ keeping the divisors $E_i$ is unobstructed. 
  \end{enumerate}
  \end{proposition}
  \begin{proof}  (i) First,  $T_{\hX}(-E) \cong T_{\hX}\otimes \Omega^n_{\hX} \cong \Omega^{n-1}_{\hX}$. Moreover, the pairing 
  $$\Omega^1_{\hX}(\log E) \otimes \Omega^{n-1}_{\hX}(\log E ) \to \Omega^n_{\hX}(\log E ) \cong\scrO_{\hX}$$
  is perfect and thus identifies $\Omega^1_{\hX}(\log E)\spcheck = T_{\hX}(-\log E)$ with $\Omega^{n-1}_{\hX}(\log E)$.
  
  \smallskip
  \noindent (ii) We have the exact sequence
\begin{equation}\label{eqn0}
0 \to T_{\hX}(-\log E) \to T_{\hX} \to \bigoplus_iN_{E_i/\hX} \to 0.
\end{equation} 
  Since $E$ is the exceptional set of $\pi$, $\bigoplus_i H^0(E_i; N_{E_i/\hX}) =0$. Thus   $H^1(\hX; T_{\hX}(-\log E)) \to H^1(\hX; T_{\hX})$ is injective and its cokernel is contained in  $\bigoplus_i H^1(E_i; N_{E_i/\hX})$. Hence $H^1(\hX; T_{\hX}(-\log E)) \to H^1(\hX; T_{\hX})$ is an isomorphism if $H^1(E_i; N_{E_i/\hX}) =0$ for every $i$.

   \smallskip
  \noindent (iii)  As in \cite[\S1]{FL22b} define the complex $\mathcal{D}^\bullet$ as the complex
  $$T_{\hX} \to N_{E/\hX}$$
  in degrees $0$ and $1$. Then $\mathbb{H}^1(\hX; \mathcal{D}^\bullet)$ is the tangent space to deformations of $\hX$ keeping the divisor $E$ and there is an exact sequence 
  $$H^0(E; N_{E/\hX}) \to \mathbb{H}^1(\hX; \mathcal{D}^\bullet) \to H^1(\hX; T_{\hX}) \to H^1(E; N_{E/\hX}).$$
  Thus, as $H^0(E; N_{E/\hX}) =0$, $\mathbb{H}^1(\hX; \mathcal{D}^\bullet) \to H^1(\hX; T_{\hX})$ is always injective, and it is surjective if $H^1(E; N_{E/\hX})= H^1(E; \scrO_E(E))=0$. 
  
   \smallskip
  \noindent (iv) By (i),  $H^2(\hX;T_{\hX}(-\log E)) \cong H^2(\hX; \Omega^{n-1}_{\hX}(\log E))$.  
 By the proof of Theorem~\ref{maintheorem}(i), 
 $$H^2(\hX; \Omega^{n-1}_{\hX}(\log E)) \cong  H^2(\hX; \Omega^{n-1}_{\hX}(\log E)|E)   = \Gr^{n-1}_FH^{n+1}(L).$$ In particular, if $X$ is a  local complete intersection, then   $\ell^{n-1, 2} =\dim H^2(\hX; \Omega^{n-1}_{\hX}(\log E)|E)=0$ by Proposition~\ref{locinvvan}(iv) and thus $H^2(\hX; \Omega^{n-1}_{\hX}(\log E)) = H^2(\hX;T_{\hX}(-\log E)) = 0$.  If $E$ is smooth, then, by Proposition~\ref{locinvvan}(iii), the map
 $$\Gr^{n-1}_FH^{n+1}(L) = H^2(\hX; \Omega^{n-1}_{\hX}(\log E)|E) \to \Gr^{n-1}_FH^{n+2}_E(\hX) = H^2(\hX; \Omega^{n-1}_{\hX}(\log E)/\Omega^{n-1}_{\hX})$$
 is injective. Since $E$ is smooth, $\Omega^{n-1}_{\hX}(\log E)/\Omega^{n-1}_{\hX} \cong \Omega^{n-2}_E$.  Thus, if $h^{n-2,2}(E) =0$, then $H^2(E; \Omega^{n-2}_E) =0$, and hence $H^2(\hX; \Omega^{n-1}_{\hX}(\log E)|E) =0$. Thus $H^2(\hX;T_{\hX}(-\log E)) = 0$ in this case as well. 
  \end{proof} 
  
  \begin{remark}\label{remark2.6} (i)  If $E$ is smooth and   $\scrO_E(-E)$ is ample, then  $H^1(E; \scrO_E(E)) =0$ and thus $E$ is stable and the hypotheses of (ii) are satisfied. More generally,   if $E$ is not necessarily smooth but  $\scrO_E(-E)$ is ample, then again  $H^1(E; \scrO_E(E)) =0$ (since $\dim E \ge 2$) and the hypotheses of (iii) are satisfied. 
  
\smallskip
  \noindent (ii)    By Proposition~\ref{Prop2.3}(i) above,   $\Omega^{n-1}_{\hX}(\log E) \cong T_{\hX}(-\log E)$. Thus the image of $H^1(\hX; \Omega^{n-1}_{\hX}(\log E))$ in   $H^0(X; T^1_X)$ is the  same as that of $H^1(\hX;T_{\hX}(-\log E))$, in other words the deformations of $\hX$ keeping the divisors $E_i$. In particular, this is a rather small subspace of $H^0(X; T^1_X)$.   Note that $H^1(\hX; \Omega^{n-1}_{\hX}(\log E))$ is a birational invariant of $X$, i.e.\ is independent of the choice of a resolution (cf.\ \cite[Remark 3.15]{FL}). Thus, if $H^1(\hX; T_{\hX}(-\log E)) \to H^1(\hX; T_{\hX})$ is an isomorphism, then the image of $ H^1(\hX; T_{\hX})$ in $H^0(X; T^1_X)$ is also independent of the choice of a resolution. 
  
   \smallskip
  \noindent (iii) In the notation of the proof of Proposition~\ref{Prop2.3}(iii), the cohomology sheaves of the complex $\mathcal{D}^\bullet$ are: $\mathcal{H}^0\mathcal{D}^\bullet = T_{\hX}(-\log E)$,  $\mathcal{H}^1\mathcal{D}^\bullet = T^1_E$, and $\mathcal{H}^i\mathcal{D}^\bullet = 0$ for $i > 1$.  Thus the second spectral sequence of hypercohomology gives an  exact sequence 
  $$H^1(\hX;T_{\hX}(-\log E)) \to \mathbb{H}^1(\hX; \mathcal{D}^\bullet) \to H^0(E; T^1_E) \to H^2(\hX;T_{\hX}(-\log E)).$$
By Proposition~\ref{Prop2.3}(iv), if  $X$ is a local complete intersection, then $\mathbb{H}^1(\hX; \mathcal{D}^\bullet) \to H^0(E; T^1_E) $ is surjective.  Additionally, if  $\scrO_E(E)$ is the dual of an ample line bundle on $E$, then 
$$H^0(E; T^1_E) = H^0(E_{\text{\rm{sing}}}; \scrO_E(E)|E_{\text{\rm{sing}}}) = 0.$$
  
   \smallskip
  \noindent (iv) As in the proof of \cite[Theorem 1.6(v)]{FL22b}, there is an exact sequence
  $$0 \to T_{\hX}(-E) \to T_{\hX}(-\log E) \to T^0_E \to 0,$$
  where $T^0_E$ is the tangent sheaf of $E$. Thus, if $H^3(\hX; T_{\hX}(-E)) = H^3(\hX; \Omega^{n-1}_{\hX}) =0$ and if $X$ is a local complete intersection so that $H^2(\hX;T_{\hX}(-\log E)) = 0$, then $H^2(E; T^0_E) = 0$. 
  
   \smallskip
  \noindent (v) In the statement of Proposition~\ref{Prop2.3}(iv), under the assumption that $E$ is smooth, the link exact sequence gives the following exact sequence:
  $$0 \to \Gr^{n-1}_FH^{n+1}(L)  \to \Gr^{n-1}_FH^{n+2}_E(\hX) \cong H^2(E; \Omega_E^{n-2}) \to \Gr^{n-1}_FH^{n+2}(E) \cong H^3(E; \Omega_E^{n-1}).$$
  The map $H^2(E; \Omega_E^{n-2}) \to H^3(E; \Omega_E^{n-1})$ is given by the cup product map $\smile ([E]|E)$, where $[E] \in H^1(\hX; \Omega^1_{\hX})$ is the fundamental class of $E$ and $[E]|E$ denotes its restriction to $H^1(E; \Omega_E^1)$. Thus $H^2(\hX;T_{\hX}(-\log E)) \cong \Gr^{n-1}_FH^{n+1}(L)=0$ $\iff$  $\smile ([E]|E)\colon H^2(E; \Omega_E^{n-2}) \to H^3(E; \Omega_E^{n-1})$ is injective. If $n=3$ and $E$ is a $K3$ surface, $h^{1,2}(E)  =0$ and hence $\smile ([E]|E)$ is automatically injective. If $n=4$, then $H^3(E; \Omega_E^3) \cong \Cee$, and so  $\smile ([E]|E)$ is injective $\iff$ $h^{2,2}(E) = 1$. For $n \ge 5$, if $E$ is a strict Calabi-Yau manifold, i.e.\ if $h^i(E; \scrO_E) = 0$ for $1\le i \le n-2$, then $H^3(E; \Omega_E^{n-1}) \cong H^3(E; \scrO_E)= 0$, and hence $\smile ([E]|E)$ is injective $\iff$ $h^{n-2,2}(E)  =0$.  
\end{remark}
  
   \subsection{Some special cases}
  By Theorem~\ref{maintheorem}(vi), it is natural to look for cases where the map $K' \to H^2_E( \hX; \Omega^{n-1}_{\hX}(\log E))$ is an isomorphism.  The first examples are concerned with dimension $3$, but with somewhat restrictive hypotheses.
  
  \begin{proposition}\label{prop2.7} If $\dim X =3$,  $X$ is  $0$-liminal, and $H^2(E_i;\scrO_{E_i}) =0$ for every $i$, then 
 $\Gr_F^2H^4_E(\hX) = H^4_E(\hX)/\im H^1_E(\hX; \Omega^3_{\hX})$ and   hence the map $K' \to H^2_E( \hX; \Omega^2_{\hX}(\log E))$ is an isomorphism.
  \end{proposition}
  \begin{proof} By hypothesis, 
  $$\Gr^1_FH^4_E(\hX)= H^2(\hX; \Omega^1_{\hX}(\log E)/ \Omega^1_{\hX}) \cong \bigoplus_iH^2(E_i;\scrO_{E_i})=0.$$ 
  Moreover $\Gr^0_FH^4_E(\hX)= 0$ since by definition $\Omega^0_{\hX}(\log E) =\scrO_{\hX}$. Thus  $H^4_E(\hX)/\im H^1_E(\hX; \Omega^3_{\hX})=\Gr^2_FH^4_E(\hX)$. By Theorem~\ref{maintheorem}(vi),  the image of   $H^2_E(\hX; \Omega^2_{\hX}) \to H^4_E(\hX)/\im H^1_E(\hX; \Omega^3_{\hX})$ contains $\Gr^2_FH^4_E(\hX)$, and hence must be equal to $\Gr^2_FH^4_E(\hX)$. Then $K' \to H^2_E( \hX; \Omega^2_{\hX}(\log E))$ is an isomorphism, again by Theorem~\ref{maintheorem}(vi).
    \end{proof} 
  
 \begin{remark}\label{remark2.8}   Roughly speaking, the condition in Proposition~\ref{prop2.7} that $H^2(E_i;\scrO_{E_i}) =0$ for every $i$ means that $E$ contains no component which is a $K3$ surface. Example~\ref{95} shows that there is a plentiful supply of examples which do not satisfy the hypotheses of Proposition~\ref{prop2.7}. Some of these are   weighted cones over $K3$ surfaces, and thus  are covered by Corollary~\ref{wtdcor} below. On the other hand, positive weight deformations of such a weighted cone will not satisfy the hypotheses of either Proposition~\ref{prop2.7} or Corollary~\ref{wtdcor}. 
  \end{remark}
  
   The conditions of Theorem~\ref{maintheorem}(vi) also hold  quite generally in  the weighted homogeneous case. First, we note the following:
   
   \begin{proposition}\label{wtdcase}  Assume that $X$ is a weighted homogeneous $0$-liminal hypersurface singularity. Then:
   \begin{align*}
   H^1(\hX; \Omega^{n-1}_{\hX}(\log E)) &= \bigoplus_{a\ge 0}H^0(X; T^1_X)(a);\\
   \Gr^{n-1}_FH^n(L) &= H^0(X; T^1_X)(0);\\
  \im  H^1(\hX; \Omega^{n-1}_{\hX}(\log E)(-E)) &=  \bigoplus_{a> 0}H^0(X;T^1_X)(a);\\
   K &\cong \bigoplus_{a\leq 0}H^0(X;T^1_X)(a);\\
  H^2_E(\hX; \Omega^{n-1}_{\hX}(\log E)) &\cong  \bigoplus_{a< 0}H^0(X;T^1_X)(a) .
   \end{align*} 
   \end{proposition} 
    \begin{proof} These follow from  \cite[Theorem 3.3, 3.5]{FL} as in the proof of \cite[Corollary 3.6]{FL}, noting that in the $0$-liminal case $N=0$. 
    \end{proof}

    \begin{corollary}\label{wtdcor}  If $X$ is a weighted homogeneous $0$-liminal hypersurface singularity, then 
    $$K \cong K'\oplus  \Gr^{n-1}_FH^n(L).$$
    Thus $K'=  \bigoplus_{a< 0}H^0(X;T^1_X)(a) $ is the negative weight space in $H^0(X;T^1_X)$ and the induced map $K' \to H^2_E( \hX; \Omega^2_{\hX}(\log E))$ is an isomorphism.
    \end{corollary}
    \begin{proof}  This follows as in   the proof of \cite[Corollary 3.6]{FL}.
    \end{proof}
   
\subsection{Some nonsmoothable examples}\label{section2.5} This paper is mainly concerned with the smoothability of generalized Calabi-Yau varieties $Y$, for which the lci assumption plays a major role.  In this section, we  consider some locally non-smoothable  examples due to  Coughlan-Sano \cite{CS} from the above point of view.  Throughout this section, we make the following assumption: $X$ is an isolated normal  singularity with resolution $\pi\colon \hX \to X$, such that the exceptional divisor $E$ is smooth, $K_{\hX}= \scrO_{\hX}(-E)$, $h^i(E;\scrO_E) = 0$ for $1\le i \le n-2$, and $\scrO_{\hX}(-E)|E$ is ample. By Lemmas~\ref{newlemma1.16} and ~\ref{goodisstrict}, $X$ is Gorenstein and $0$-liminal. Our goal is to show the following:

\begin{theorem}\label{nonsmooththm}  Under the above assumptions, suppose that $H^{n-2}(\hX; \Omega^1_{\hX}(\log E)(-E)) = 0$ and either that $h^{n-2,2}(E)  =0$ or that $n=4$ and $b_2(E) =1$. Then the local singularity $(X,x)$ is not smoothable. 
\end{theorem}
\begin{proof}  Let    $\mathbf{Def}_{(\hX, E)}$ be the local deformation functor of deformations of the pair $(\hX, E)$.  In the case where $X =\Spec R$ is an affine variety, and hence   $\hX$ is a scheme (since the morphism $\hX \to X$ is projective), the  natural morphism of functors   $\mathbf{Def}_{(\hX, E)} \to \mathbf{Def}_X$ is given as follows: Given a deformation $(\hX_A ,E_A) \to \Spec A$ of $(\hX,E)$ over the Artin local $\Cee$-algebra, set $X_A= \Spec H^0(\hX_A ; \scrO_{\hX_A })$, an affine scheme over $\Spec A$ which is flat over $\Spec A$ by \cite[Proposition 2.7]{Wahl}. In particular, there is a morphism of $A$-schemes $\hX_A \to X_A$. 
In the analytic case, where $X$ is a Stein space, there is a similar construction   using the standard result that a  Stein space $X$ is determined by the ring $H^0(X; \scrO_X)$. On the level of Zariski tangent spaces, the morphism $\mathbf{Def}_{(\hX, E)} \to \mathbf{Def}_X$ induces a homomorphism $H^1(\hX; T_{\hX}(-\log E)) \to H^0(X; T^1_X)$. By Theorem~\ref{maintheorem}(i), there is an exact sequence
$$\begin{CD} 
0 @>>>  H^1(\hX; \Omega^{n-1}_{\hX}(\log E)) @>>>  H^1(U; T_U) @>>> H^2_E(\hX; \Omega^{n-1}_{\hX}(\log E))\\
@. @VV{\cong}V  @VV{\cong}V  @. \\
@.   H^1(\hX; T_{\hX}(-\log E))  @>>> H^0(X; T^1_X) @. 
\end{CD}$$
By assumption, $H^{n-2}(\hX; \Omega^1_{\hX}(\log E)(-E)) = 0$, and hence $H^2_E(\hX; \Omega^{n-1}_{\hX}(\log E)) =0$ by duality. Thus the natural map $H^1(\hX; T_{\hX}(-\log E)) \to H^0(X; T^1_X)$ is an isomorphism. 

By Proposition~\ref{Prop2.3}(iv) and Remark~\ref{remark2.6}(v), $H^2(\hX; T_{\hX}(-\log E)) =0$. (Note that, if $n=4$, by our assumptions on $E$, $h^{2,2}(E) =  b_4(E)= b_2(E)$.)   By \cite[Proposition 3.4.17]{Sernesi},   $\mathbf{Def}_{(\hX, E)}$ is unobstructed, i.e.\ there is a semi\-uni\-versal deformation whose base is a formal power series ring of dimension $\dim H^1(\hX; T_{\hX}(-\log E))$. It   follows from \cite[2.3.7]{Sernesi} that the natural morphism of functors $\mathbf{Def}_{(\hX, E)} \to \mathbf{Def}_X$ is smooth and hence surjective. It is injective as well, and therefore  an isomorphism, by
 \cite[Proposition 2.5(ii)]{Wahl}, since
$$H^1_E(\hX; T_{\hX}(-\log E)) = H^1_E(\hX; \Omega^{n-1}_{\hX}(\log E) ) =0,$$
by Proposition~\ref{Prop2.3}(i), and Lemma~\ref{extravan}.  Thus every deformation of $X$ is realized by a deformation of the pair $(\hX, E)$. In particular, such deformations are equisingular in the sense of \cite[2.4]{Wahl}, and do not correspond to smoothings of $X$. 

Concretely, assume by contradiction that there is a one parameter smoothing $\mathcal{X}\to \Delta$ of $X$ over the disk. By a theorem of Bingener and Kosarew \cite[(V.5.2)]{BingKos}, there is a germ $(\widehat{\mathcal{X}}, \mathcal{E})$ of a smooth analytic space $\widehat{\mathcal{X}}$ containing a smooth hypersurface $\mathcal{E}$ along with a morphism  $\widehat{\mathcal{X}} \to \Delta$ such that  the induced  morphism  $\mathcal{E} \to \Delta$ is proper, together with an induced    morphism $\widehat{\mathcal{X}} \to \mathcal{X}$  which is a deformation of the morphism $\hX \to X$ (possibly after shrinking $\Delta$).    For $t\in \Delta$, $t\neq 0$, the fiber $\mathcal{X}_t$ over $t$ is smooth and the corresponding degree one morphism $\widehat{\mathcal{X}}_t \to \mathcal{X}_t$ contracts $\mathcal{E}_t$ to a point. However, since $H^{n-1}(\mathcal{E}_t, \scrO_{\mathcal{E}_t}) = H^{n-1}(E; \scrO_E) \neq 0$, this is not possible. Hence $X$ is not smoothable.
\end{proof}

\begin{remark}  There is a more algebraic version of the above argument in case $X=\Spec R$ is affine and hence $\hX$ is a scheme. Suppose that there is a one parameter smoothing of $X$ corresponding to a formal deformation $\mathfrak{X} $. By the above discussion there is then a formal deformation $(\widehat{\mathfrak{X} },  \mathfrak{E} )$ of the pair $(\hX, E)$  and a proper morphism of formal schemes $\widehat{\mathfrak{X} } \to \mathfrak{X}$. By a theorem of Elkik, there exists a scheme $\mathfrak{X}^{\text{\rm{alg}}}=\Spec A$ over $\Spec \Cee[[t]]$ inducing the formal scheme $ \mathfrak{X}$ and such that $\mathfrak{X}^{\text{\rm{alg}}}\times_{\Spec \Cee[[t]]}\Spec \Cee((t))$ is regular. After replacing $A$ by its $(t)$-adic completion $A\,\hat{}$,   the pair  $(\widehat{\mathfrak{X} },  \mathfrak{E} )$ is then algebraizable over $A\,\hat{}$ by Grothendieck's existence theorem \cite[Th\'eor\`eme 5.4.5]{EGAIII} or \cite[\href{https://stacks.math.columbia.edu/tag/089A}{Tag 089A}]{stacks-project} because the sheaf $\scrO_{\hX}(-E)$ is relatively ample for the morphism $\hX \to X$. Then we reach a contradiction in much the same way as before. (We  are grateful to Johan de Jong for sketching a proof of  this argument to us.)
\end{remark} 

\begin{example} Let $X$ be the germ (at the vertex)  of the cone over a quartic $K3$ surface $E$ embedded in projective space via the line bundle $\scrO_E(d)$, where $\scrO_E(1)$ corresponds to the embedding of $E$ in $\Pee^3$.  It follows from \cite[Example 5.12 and Remark 5.13]{SVV} that, if $d \ge 5$, then $H^1(\hX; \Omega^1_{\hX}(\log E)(-E)) = 0$ and in fact   $X$ is $1$-Du Bois. Note however that $X$ is not a rational singularity (cf.\ the proof of Theorem~\ref{maintheorem}(vii) where the lci assumption is crucial). By Theorem~\ref{nonsmooththm},  the germ  $X$ is not smoothable, recovering some of the examples of \cite{CS}. Similar arguments using \cite[Proposition 5.10]{SVV} give nonsmoothable examples in all dimensions $\ge 3$. 
\end{example}

\section{The global case}\label{section3} 

Throughout this section,  $Y$ denotes a compact analytic variety  with isolated singularities, and  $Z = Y_{\text{\rm{sing}}}$. Let  $\pi\colon \hY \to Y$ be  a good equivariant    resolution with exceptional divisor $E$, and set  $V = Y-Z = \hY -E$.  We will also freely use the global analogue of a result of Schlessinger: if $Y$ has depth at least $3$ at each singular point, then $\mathbb{T}^1_Y = H^1(V, T_V) $ \cite[Lemma 4.1]{FL}. 

\subsection{Preliminaries} We begin by defining the class of varieties for which our results will apply.

\begin{definition}\label{definestrict} A \textsl{strict $0$-liminal  Calabi-Yau variety} of dimension $n\ge 3$ is a compact Gorenstein analytic space of dimension $n$ with a unique singular point $x$,  satisfying the following: 
\begin{enumerate}
\item[(i)]   $\omega_Y \cong \scrO_Y$.
\item[(ii)] $Y$ has a unique singular point, which is a $0$-liminal local complete intersection singularity; 
\item[\rm(iii)] $H^i(Y; \scrO_Y) = 0$ for $0< i< n$.
\item[\rm(iv)] There exists a  log resolution $\pi\colon \hY \to Y$  satisfying the $\partial\bar\partial$-lemma.
 \item[\rm(v)] There exists  a log resolution   $\pi\colon \hY \to Y$ such that  $H^i(\hY; \scrO_{\hY}) = 0$   for all $i> 0$.  
\end{enumerate}
\end{definition}

\begin{remark}\label{rmkstrict} (i) It is easy to see that if (iv) and (v) hold for one log resolution, then they hold for all log resolutions.

\smallskip

\noindent (ii) Suppose that $Y$ is a compact analytic space satisfying (i)   and (iv), such that $Y$ has isolated Du Bois singularities,  and that  there exists a smoothing of $Y$ to a strict Calabi-Yau manifold $Y_t$ (i.e.\ $H^i(Y_t; \scrO_{Y_t}) = 0$ for $0< i< n$.  By a theorem of Du Bois \cite[Theorem 4.6]{duBois} (cf.\ also Du Bois-Jarraud \cite{DBJ}),   $Y$ satisfies (iii) above as well. More generally, assume that the singularities of $Y$ are  Du Bois but not necessarily isolated and that $Y$ is projective, or has a K\"ahler resolution, so that (iv) is automatic.  If there exists a smoothing of $Y$ to a strict Calabi-Yau manifold $Y_t$ as above, then (iii) holds as well.

\smallskip

\noindent (iii) If $Y$ satisfies   (ii) above, then by Lemma~\ref{lemma1.1} and Theorem~\ref{Thm1.1},  $R^i\pi_*\scrO_{\hY} = 0$ for $0< i < n-1$ and $R^{n-1}\pi_*\scrO_{\hY} $ has length $1$. In particular, the only possible nonzero differential in the Leray spectral sequence $H^p(Y; R^q\pi_*\scrO_{\hY}) \implies H^{p+q}(\hY; \scrO_{\hY})$ is
$$d_n \colon H^0(Y;  R^{n-1}\pi_*\scrO_{\hY})\to H^n(Y; \scrO_Y).$$
Here both sides are one-dimensional vector spaces. Thus,   (v) is equivalent to the condition that  $d_n \colon H^0(Y;  R^{n-1}\pi_*\scrO_{\hY})\to H^n(Y; \scrO_Y)$ is an isomorphism. 

 \smallskip

\noindent (iv) For simplicity in this paper, we stick to the case of a unique singular point. It would certainly be of interest to allow additional singular points which are rational. Allowing more than one nonrational singular point might be of somewhat more limited interest. For example, in dimension $2$, a generalized Calabi-Yau surface $Y$ can have at most two $0$-liminal singularities. In fact, we have the following:
 \end{remark} 

\begin{lemma}\label{classdim2}  Let $Y$ be a normal compact analytic surface with $\omega_Y \cong \scrO_Y$,    $k>0$ simple elliptic or cusp singularities,  and all remaining singularities rational double points. Then $k \le 2$. Moreover, if $\hY$ is  the minimal resolution  of $Y$ and  $\hY$   is algebraic, then: 
\begin{enumerate}
\item[\rm(i)]   $\hY$ is a rational surface $\iff$  $k=1$. 
\item[\rm(ii)]   $\hY$ is a (nonminimal) elliptic ruled surface $\iff$  $k=2$, and in this case both singularities are simple elliptic singularities. 
\end{enumerate}
\end{lemma} 
\begin{proof}    By assumption  $K_{\hY} = \scrO_{\hY}(-\sum_{i=1}^kE_i)$, where the $E_i$ are the exceptional fibers of the elliptic singularities in the resolution $\hY \to Y$. Since $k\ge 1$,  the Kodaira dimension $\kappa(\hY)$ is $-\infty$. By the classification of compact analytic surfaces,   $\hY$ is rational, ruled over a curve of genus $g\ge 1$, or of Type VII.  A Leray spectral sequence argument implies that $\chi(\scrO_{\hY}) = \chi(\scrO_Y)-k$. Since $ \chi(\scrO_Y) \le 2$, $\chi(\scrO_{\hY})  \le 2-k \le 1$, and $\chi(\scrO_{\hY}) \le 0$ if $k \ge 2$. In the Type VII case, $\chi(\scrO_{\hY}) = 0$ and hence $k\le 2$. If $\hY$ is ruled over a curve of genus $g\ge 1$, then $\chi(\scrO_{\hY}) = 1-g$. Since $\hY$ contains the curves $E_i$, if it is ruled over a curve $C$ of genus $g\ge 1$, then all singularities are simple elliptic, $g=1$,  the morphisms $E_i \to C$ must be   isogenies, and $\chi(\scrO_{\hY}) =0$. Thus $k \le 2$ in this case as well. A more careful analysis shows that it is not possible for there to be a smooth bisection $\Gamma$ of the morphism $\hY \to C$ such that $K_{\hY} =\scrO_{\hY}(-\Gamma)$ (though it is possible for $-K_{\hY}$ to be numerically equivalent to a smooth bisection); we omit the proof. Thus, if $\hY$ is an elliptic ruled surface, then $k=2$. If $\hY$ is rational, then $\chi(\scrO_{\hY}) = 1$. Hence $k\le 1$ in this case, and so  $k=1$.
\end{proof} 

\begin{remark}  (i) In  the situation of Lemma~\ref{classdim2} , if the minimal resolution of $Y$ is a   Hirzebruch-Inoue surface and hence is not algebraic, then $k=2$ and   both of the singularities are cusps. While this case is interesting from the viewpoint of smoothing cusp singularities,  surfaces $Y$ with $\hY$ algebraic and $k=2$  are somewhat more peripheral to the study of degenerations of $K3$ surfaces than  Calabi-Yau surfaces with one $0$-liminal  singularity, whose minimal resolutions are rational surfaces. Thus, the assumption of a unique $0$-liminal singularity seems reasonable in general.    (See however \cite[Example 1.14, Example 1.20]{FG24} for an analysis of the deformation theory in the case where $Y$ has two simple elliptic singularities and $\hY$ is an elliptic ruled surface.)

 \smallskip

\noindent (ii) One can  construct generalized Calabi-Yau varieties $Y$  in all dimensions $n \ge 2$ with  two $0$-liminal singularities, by starting with a ruled variety $R=\Pee(\scrO_Z\oplus L)$ over a strict Calabi-Yau variety $Z$ of dimension $n-1$ and making appropriate blowups along smooth subvarieties of the two distinguished sections  $W_1$, $W_2$ so that the proper transforms of the sections become exceptional. For example, if $L$ is ample, then one of the sections, say $W_1$, is already exceptional. To deal with $W_2$, let $A$ be a smooth hypersurface in $W_2$ and let $\hY\to R$ be the blowup of the  codimension two subvariety $A$ of $R$. Then the normal bundle of the proper transform $W_2'$ of $W_2$ in $\hY$ satisfies: $N_{W_2'/\hY} \cong N_{W_2/R} \otimes \scrO_{W_2}(-A)$ (under the natural identification of $W_2'$ and $W_2$).  Thus, as long as $A$ is sufficiently ample, $N_{W_2'/\hY}$ is the dual of an ample line bundle on $W_2'$ and hence by Grauert's criterion the divisor $W_2'\subseteq \hY$ can be contracted to a point. By Lemma~\ref{newlemma1.16} and Lemma~\ref{goodisstrict},  if  $Y$ denotes the normal analytic space obtained by contracting the two divisors  $W_1, W_2' \subseteq \hY$ to points, then the two singular points of $Y$ are Gorenstein and $0$-liminal. With a little more care, one can also arrange for projective generalized Calabi-Yau varieties $Y$   with  two $0$-liminal singularities.
\end{remark}

We have the following lemma on $\mathbb{T}^0_Y = H^0(Y; T^0_Y)$, the tangent space to $\Aut Y$:

\begin{lemma}\label{T0} If $Y$ is just assumed to have isolated local complete intersection singularities and $\pi\colon \hY \to Y$ is an equivariant resolution, then $H^0(V; T_V) =  H^0(Y; T^0_Y) \cong H^0(\hY; T_{\hY})$. 
\end{lemma}
\begin{proof} Since $T^0_Y$ has depth at least $2$ by \cite[Lemma 1.15]{FL}, $H^0(V; T_V) =  H^0(Y; T^0_Y)$. By assumption, $\pi_*T_{\hY} = T^0_Y$, hence $H^0(Y; T^0_Y) \cong H^0(Y;  \pi_*T_{\hY}) \cong H^0(\hY; T_{\hY})$ by the Leray spectral sequence.
\end{proof} 

We also note the following definition from  \cite[Definition 1.10]{FL}:

\begin{definition}\label{defstsm} Let $(X,x)$ be the germ of an isolated hypersurface singularity, so that $T^1_{X,x} \cong \scrO_{X,x}/J$ is a cyclic $\scrO_{X,x}$-module. An element  $\theta\in T^1_{X,x}$ is a \textsl{first order smoothing} if $\theta\notin \mathfrak{m}_xT^1_{X,x}$. If $Y$ is a compact analytic space with only isolated hypersurface singularities, then a first order deformation $\theta\in \mathbb{T}^1_Y$ is a \textsl{first order smoothing} if the image $\theta_x$ of $\theta$ in $T^1_{Y,x}$ is a  first order smoothing for every $x\in Y_{\text{\rm{sing}}}$. A standard argument (e.g.\ \cite[Lemma 1.9]{FL}) shows that,  if  $f\colon \mathcal{Y}\to \Delta$ is a deformation of $Y$ over the disk, then its Kodaira-Spencer class $\theta$ is a  first order smoothing $\iff$ $\mathcal{Y}$ is smooth.  In particular, in this case,   the nearby fibers $Y_t =f^{-1}(t)$, $0< |t| \ll 1$, are smooth.
\end{definition} 

\subsection{Global geometry of a resolution}\label{subsection3.2}   Now assume that $Y$ is a compact analytic variety  with isolated Du Bois singularities and that $\pi \colon \hY\to Y $ is a good resolution.  As noted in the proof of Theorem~\ref{maintheorem}(i), by the Du Bois condition, $H^1_E(\hY; \Omega^{n-1}_{\hY}(\log E)) = 0$.  Thus there is an exact sequence
$$ 0\to  H^1(\hY; \Omega^{n-1}_{\hY}(\log E)) \to  H^1(V; T_V) \to H^2_E(\hY; \Omega^{n-1}_{\hY}(\log E)) \to  H^2(\hY; \Omega^{n-1}_{\hY}(\log E)).$$
 If in addition the singularities are lci, then $H^2_E(\hY; \Omega^{n-1}_{\hY}(\log E))\cong H^2_E(\hX; \Omega^{n-1}_{\hX}(\log E))$ has dimension $b^{1, n-2} = s_1$ by  Theorem~\ref{maintheorem}(i).

To analyze the subspace $H^1(\hY; \Omega^{n-1}_{\hY}(\log E))$, we first consider for simplicity  the case where $E$ is smooth. In this case we have the Poincar\'e residue sequence 
\begin{equation}\label{2.99}
0 \to \Omega^{n-1}_{\hY} \to \Omega_{\hY}^{n-1}(\log E) \to \Omega^{n-2}_E \to 0.
\end{equation}
For the map on $H^1$ terms we get  
$$H^0(E; \Omega^{n-2}_E) \xrightarrow{\gamma} H^1(\hY; \Omega^{n-1}_{\hY} ) \to H^1(\hY;\Omega_{\hY}^{n-1}(\log E)) \to H^1(E; \Omega^{n-2}_E) \xrightarrow{\gamma}   H^2(\hY;\Omega^{n-1}_{\hY}).
$$
Here  $\gamma \colon H^i(E; \Omega^{n-2}_E) \to H^{i+1}(\hY;\Omega^{n-1}_{\hY})$ is the Gysin map.  In particular, there is an exact sequence 
$$0\to H^1(\hY; \Omega^{n-1}_{\hY} )/\im \gamma  \to H^1(\hY;\Omega_{\hY}^{n-1}(\log E)) \to \Ker \gamma \to 0.$$
The kernel of $\gamma$ on  $H^1(E; \Omega^{n-2}_E)$ is the annihilator of the image of the dual map $\iota^*\colon H^{n-2}(\hY;\Omega^1_{\hY}) \to H^{n-2}(E; \Omega^1_E)$.  Moreover, 
 there is an exact sequence
$$0 \to \Coker \gamma \to  H^2(\hY;\Omega_{\hY}^{n-1}(\log E)) \to \Ker\{ H^2(E; \Omega^{n-2}_E) \to  H^3(\hY;\Omega^{n-1}_{\hY})\}.$$
In case $E=\bigcup_iE_i$ is a union of divisors with normal crossings, the weight spectral sequence gives   similar exact sequences, where for example we replace the kernel of $\gamma$ on  $H^1(E; \Omega^{n-2}_E)$ by
$$\Ker\left\{\bigoplus_iH^1(E_i; \Omega^{n-2}_{E_i}) \to  H^2(\hY;\Omega^{n-1}_{\hY})\right\}\Big/ \im\left\{\bigoplus_{\#(I) = 2}H^0(E_I; \Omega^{n-3}_{E_I}) \to \bigoplus_iH^1(E_i; \Omega^{n-2}_{E_i})\right\},$$
where the $E_I$ are as defined in \S\ref{The exceptional divisor} and the maps are Gysin maps or alternating sums of Gysin maps.

\subsection{Unobstructedness in dimension $3$} Our goal here is to show that the unobstructedness theorem of Namikawa~\cite[Theorem 1]{namtop} for dimension $3$ holds also in the case of $0$-liminal singularities satisfying Definition~\ref{definestrict} along with a mild additional assumption. 

\begin{theorem}\label{dim3unob} Suppose that $Y$ is a strict $0$-liminal Calabi-Yau threefold   and   $H^0(Y; T^0_Y) = 0$, i.e.\   $\Aut Y$ is discrete. Then $\mathbf{Def}_Y$ is unobstructed.
\end{theorem}
\begin{proof} The proof of \cite[Theorem 1]{namtop} shows that, if the singularities of $Y$ are local complete intersections,  then $\mathbf{Def}_Y$ is unobstructed provided that the following hold:
\begin{enumerate}
\item $H^1(Y;\scrO_Y) = H^2(Y;\scrO_Y) = 0$.
\item The map
$$H^0(Y; R^1\pi_*\scrO_{\hY}^*)\otimes_{\Zee}\Cee  \xrightarrow{\delta\otimes 1} H^0(Y; R^1\pi_*\Omega^1_{\hY})$$
is injective. 
\item $H^0(Y; T^0_Y) = 0$, or equivalently (by Serre duality) $H^3(Y; \Omega^1_Y) = 0$.
\end{enumerate}
Here,  (1) is Condition (iii) in Definition~\ref{definestrict}, (2) is Lemma~\ref{namlemma}, and (3) is one of the  hypotheses. Thus $\mathbf{Def}_Y$ is unobstructed. 
\end{proof}

\begin{remark} The proof of Theorem~\ref{dim3unob} goes through under slightly weaker hypotheses than the assumption that $Y$ is a strict $0$-liminal Calabi-Yau threefold. \end{remark}

\subsection{Local versus global: constructing first order smoothings} For the rest of this section, we assume that $Y$ is a strict $0$-liminal Calabi-Yau variety.   Let $X$ be a contractible Stein neighborhood of the singular point of $Y$ and let $\pi\colon \hY \to Y$ be a good equivariant resolution.  Setting $\hX=\pi^{-1}(X)$, we also denote the resolution $\hX \to X$ by $\pi$.
 The main technical result is the following:

\begin{lemma}\label{lemma2.3}  There is an induced map $H^2_E(\hY; \Omega^{n-1}_{\hY}) \to H^{n+1}(\hY)$. Via the excision isomorphism $H^2_E(\hY; \Omega^{n-1}_{\hY})  \cong H^2_E(\hX; \Omega^{n-1}_{\hX}) $, the subspace $K'$ of Theorem~\ref{maintheorem}(vi) maps to $0$ in $H^2(\hY; \Omega^{n-1}_{\hY})$.
\end{lemma} 
\begin{proof} By Serre duality, $ H^1(\hY; \Omega^n_{\hY}) $ is dual to $H^{n-1}(\hY; \scrO_{\hY})$ and hence is $0$ by assumption. Since the Hodge-de Rham spectral sequence degenerates at $E_1$, there  is   an injective map $H^2(\hY; \Omega^{n-1}_{\hY})\to  H^{n+1}(\hY)$. Via the map $H^{n+1}_E(\hY) \to H^{n+1}(\hY)$, the image of  $H^1_E(\hY; \Omega^n_{\hY})$ is $0$, hence as in the proof of Theorem~\ref{maintheorem}(v) there is an induced map 
$$H^2_E(\hY; \Omega^{n-1}_{\hY}) \to H^{n+1}_E(\hX)/\im H^1_E(\hY; \Omega^n_{\hY}) \to H^{n+1}(\hY).$$

There is  thus a commutative diagram
$$\begin{CD}
H^2_E(\hX; \Omega^{n-1}_{\hX}) @>{\cong}>> H^2_E(\hY; \Omega^{n-1}_{\hY})  @>>> H^2(\hY; \Omega^{n-1}_{\hY})\\
@VVV @VVV @VVV\\
H^{n+1}_E(\hX) /\im H^1_E(\hX; \Omega^n_{\hX}) @>{\cong}>> H^{n+1}_E(\hY)/\im H^1_E(\hY; \Omega^n_{\hY}) @>>> H^{n+1}(\hY).
\end{CD}$$
 By definition $K' =\Ker\{H^2_E(\hX; \Omega^{n-1}_{\hX}) \to H^{n+1}_E(\hX)/\im H^1_E(\hX; \Omega^n_{\hX})\}$, and hence its image in  $H^2_E(\hY; \Omega^{n-1}_{\hY}) $ maps to $0$ in $H^{n+1}(\hY)$. Then the image of $K'$ in  $H^2(\hY; \Omega^{n-1}_{\hY})$ must be  $0$ since the map $H^2(\hY; \Omega^{n-1}_{\hY})\to  H^{n+1}(\hY)$ is injective.
\end{proof}

 Lemma~\ref{lemma2.3} now implies  the following theorem on existence of  first order smoothings (Definition~\ref{defstsm}):

\begin{theorem}\label{smoothable}  Suppose that   $Y$ is a strict $0$-liminal Calabi-Yau variety whose unique singular point is a hypersurface  singularity satisfying the assumption of Theorem~\ref{maintheorem}(vi): the map $K' \to H^2_E( \hX; \Omega^2_{\hX}(\log E))$ is an isomorphism. Then the map $\mathbb{T}^1_Y \to K'$ given as the composition of the maps
$$\mathbb{T}^1_Y \to H^0(X; T^1_X)  \to K= K'\oplus \Gr^{n-1}_FH^n(L) \to K'.$$
is surjective.  Moreover, a  first order smoothing of $Y$ exists.
\end{theorem}
\begin{proof} As in the proof of \cite[Lemma 5.6(iii)]{FL}, with $V = Y-\{x\} = \hY -E$, we have a commutative diagram
$$\begin{CD}
\mathbb{T}^1_Y = H^1(V; \Omega^{n-1}_{\hY}|V) @>>>  H^2_E(\hY; \Omega^{n-1}_{\hY}) @>>>  H^2(\hY; \Omega^{n-1}_{\hY})\\
@VVV @VV{\cong}V @.\\
H^0(X; T^1_X) = H^1(U; \Omega^{n-1}_{\hX}|U) @>>> H^2_E(\hX; \Omega^{n-1}_{\hX}) @.  
\end{CD}$$
where the top line comes from the local cohomology sequence.

By Lemma~\ref{lemma2.3}, the image of $K'$ is in the kernel of the map $H^2_E(\hY; \Omega^{n-1}_{\hY}) \to H^2(\hY; \Omega^{n-1}_{\hY})$, and hence is in the image of $\mathbb{T}^1_Y$.    Let $1' \in K'$ be the image of a generator $1$ of the cyclic module $H^0(X; T^1_X)$. Then there exists a $\theta\in  \mathbb{T}^1_Y$ whose image is $1'$.  If $\bar\theta$ is  the image of $\theta$ in $H^0(X; T^1_X)$, then, using Proposition~\ref{prop1.12}, 
$$\bar\theta - 1\in \Ker\{H^0(X; T^1_X) \to K'\} \subseteq \mathfrak{m}_xH^0(X; T^1_X).$$
 Hence $\bar\theta \notin \mathfrak{m}_xH^0(X; T^1_X)$, so by definition $\theta$ is a first order smoothing of $Y$.
\end{proof}

\begin{corollary} Suppose that $Y$ is a strict $0$-liminal Calabi-Yau threefold whose unique singular point   is a hypersurface singularity which is either  weighted homogeneous  or satisfies  $H^2(E_i;\scrO_{E_i}) =0$ for every $i$. Then $Y$ is smoothable. \qed
\end{corollary}

\begin{remark} (i) As noted in Remark~\ref{remark2.8}, the hypotheses above do not cover all possible cases.   It seems very likely that the corollary holds more generally, and indeed that the hypotheses of Theorem~\ref{smoothable}  hold in greater generality.

\smallskip 
\noindent (ii) In case the singular point of $Y$ is a weighted homogeneous hypersurface singularity, Theorem~\ref{smoothable} shows that  the projection of the image of 
$\mathbb{T}^1_Y$ in $H^0(X; T^1_X)$  to the negative weight space is surjective. Note however that this does not imply that the image of $\mathbb{T}^1_Y$ in $ H^0(X; T^1_X)$ contains the negative weight space.  
\end{remark}

\section{The good anticanonical case}\label{section4} 

In the good anticanonical case, the pair $(\hY,E)$ is by definition a log Calabi-Yau variety.   More generally, let  $\hY$ be an arbitrary log Calabi-Yau manifold, i.e.\ a smooth projective manifold of dimension $n$ with $E =\bigcup_iE_i\in |-K_{\hY}|$ a normal crossing divisor, not necessarily exceptional. We first note the following result of Katzarkov-Kontsevich-Pantev\cite{KKP} and Iacono \cite[Remark 5.6]{Iacono}:

\begin{theorem}\label{Iacono} With assumptions as above, the functor  $\mathbf{Def}_{(\hY;E_1, \dots, E_r)}$  of locally trivial deformations of the pair $(\hY, E)$ is unobstructed.  \qed
\end{theorem}

Moreover, $E$ is itself a Calabi-Yau manifold of dimension $n-1$ or more generally a normal crossing variety with trivial dualizing sheaf, and we can consider the relationships among the deformations of the pair $(\hY,E)$ and deformations of $E$, or the locally trivial deformations of  $(\hY,E)$ and locally trivial deformations of $E$. Finally, it might be possible to complete $(\hY,E)$ to a $d$-semistable normal crossing variety   of dimension $n$ with trivial dualizing sheaf, and the deformation theory of such varieties has been studied by  Kawamata-Namikawa \cite{KawamataNamikawa}. In this section, we make some very preliminary comments on the relationships among these various deformation functors. 

\subsection{Deformations of the pair $(\hY,E)$}\label{subsection3.5}  First consider the locally trivial deformations of the pair $(\hY; \sum_iE_i)$. Here $T_{\hY}(-\log E) \cong \Omega_{\hY}^{n-1}(\log E)$, and  we have the global analogue of the exact sequence \ref{eqn0}:
$$0 \to T_{\hY}(-\log E) \to T_{\hY} \to \bigoplus_iN_{E_i/\hY} \to 0.$$
In particular, by the argument of Proposition~\ref{Prop2.3}(ii) and Theorem~\ref{Iacono}, we have:

\begin{proposition}\label{prop4.2}  Suppose that  $\hY$ is a good anticanonical resolution of the strict $0$-liminal Calabi-Yau variety $Y$ and $E$ is the exceptional set. The natural map $H^1(\hY; T_{\hY}(-\log E)) \to H^1(\hY; T_{\hY})$ is injective, and it is surjective if $H^1(E_i; N_{E_i/\hY}) =0$ for every $i$.  Thus, if  $H^1(E_i; N_{E_i/\hY}) =0$ for every $i$, then the  forgetful map $\mathbf{Def}_{(\hY;E_1, \dots, E_r)}  \to \mathbf{Def}_{\hY}$ is an isomorphism and hence $\mathbf{Def}_{\hY}$ is unobstructed.\qed
\end{proposition}

\begin{remark} The assumption that $H^1(E_i; N_{E_i/\hY}) =0$ for every $i$ seems to be a fairly mild one. For example, if $E$ is smooth and $\scrO_E(-E)$ is ample, then $H^1(E; N_{E/\hY}) =0$ by Kodaira vanishing.
\end{remark}

As noted in the proof of Theorem~\ref{maintheorem}(i), under the hypotheses of Proposition~\ref{prop4.2}, by the Du Bois condition, $H^1_E(\hY; \Omega^{n-1}_{\hY}(\log E)) = 0$.  Thus there is a commutative diagram
$$\begin{tikzcd}[column sep=small]
 \scriptstyle 0  \arrow[r]  &  \scriptstyle H^1(\hY; \Omega^{n-1}_{\hY}(\log E)) \arrow[r] \arrow[d, "\cong"] &  \scriptstyle H^1(V; T_V) \arrow[r] \arrow[d, "\cong"] &   \scriptstyle H^2_E(\hY; \Omega^{n-1}_{\hY}(\log E)) \arrow[r] \arrow[d, "\cong"] &   \scriptstyle H^2(\hY; \Omega^{n-1}_{\hY}(\log E)) \arrow[d, "\cong"]\\
 \scriptstyle 0\arrow[r]  &  \scriptstyle H^1(\hY; T_{\hY}(-\log E)) \arrow[r]  &  \scriptstyle \mathbb{T}^1_Y \arrow[r]  &   \scriptstyle H^2_E(\hY; T_{\hY}(-\log E))\arrow[r]  &   \scriptstyle H^2(\hY; T_{\hY}(-\log E)).
 \end{tikzcd}$$

For simplicity, from now on we   restrict ourselves to the case $r=1$, i.e.\ $E$ is smooth. In general, we have the exact sequence
\begin{equation}\label{3.00}
0 \to T_{\hY}(-E) \to T_{\hY}(-\log E) \to T_E \to 0.
\end{equation}
In our case, $T_{\hY}(-E) = T_{\hY} \otimes K_{\hY} \cong \Omega^{n-1}_{\hY}$, $T_{\hY}(-\log E) \cong \Omega_{\hY}^{n-1}(\log E)$, and $T_E \cong \Omega^{n-2}_E$ since $K_E\cong \scrO_E$ and $E$ has dimension $n-1$. Thus the  sequence (\ref{3.00}) is the Poincar\'e residue sequence (\ref{2.99}).
For the map on $H^1$ terms we get (all vertical maps isomorphisms) 
$$\begin{CD}
H^1(\hY;T_{\hY}(-E)) @>>> H^1(\hY; T_{\hY}(-\log E)) @>>> H^1(E; T_E) @>{\partial}>> H^2(\hY;  T_{\hY}(-E)) \\
@VVV @VVV @VVV @VVV \\
H^1(\hY; \Omega^{n-1}_{\hY} ) @>>> H^1(\hY;\Omega_{\hY}^{n-1}(\log E)) @>>> H^1(E; \Omega^{n-2}_E) @>{\gamma}>>  H^2(\hY;\Omega^{n-1}_{\hY}).
\end{CD}$$
Then the coboundary map $\partial \colon H^1(E;T_E) \to H^2(\hY;  T_{\hY}(-E))$ is identified with the Gysin map $\gamma \colon H^1(E; \Omega^{n-2}_E) \to H^2(\hY;\Omega^{n-1}_{\hY})$. The kernel of $\partial$ is the annihilator of the image of the dual map $\iota^*\colon H^{n-2}(\hY;\Omega^1_{\hY}) \to H^{n-2}(E; \Omega^1_E)$.  Summarizing, we have the following:

\begin{proposition}\label{prop4.4} With $\partial =\gamma$ and $\iota$   as above,
 the image of $H^1(\hY;T_{\hY}(-\log E))$ in $H^1(E;T_E)$ is equal to
$$\{\theta \in  H^1(E;T_E): \theta \smile \iota^*\xi  = 0 \text{ for all $\xi \in H^{n-2}(\hY;\Omega^1_{\hY})$}\}. \qed$$
\end{proposition}

The meaning of the proposition is as follows: Let $\mathbf{Def}_E$ be the deformation functor of $E$ and let $\mathbf{Def}_{(\hY,E)}$ be the functor  of deformations of the pair $(\hY, E)$. Then  $\mathbf{Def}_E$ is unobstructed by the Tian-Todorov theorem, and  $\mathbf{Def}_{(\hY,E)}$  of deformations of the pair $(\hY, E)$ is unobstructed by Theorem~\ref{Iacono}. The map $H^1(\hY; T_{\hY}(-\log E)) \to  H^1(E; T_E)$ is the corresponding map on tangent spaces.

For example, suppose that  $n=3$. If  $H^1(E; \scrO_E) =0$, then $E$ is a $K3$ surface. If $H^1(\hY; \scrO_{\hY}) =0$, then $H^1(E; \scrO_E) =0$ as follows from the exact sequence
$$H^1(\hY; \scrO_{\hY}) \to H^1(E; \scrO_E) \to H^2(\hY; \scrO_{\hY}(-E))= H^2(\hY; K_{\hY} )$$
and duality. In the following, we make the assumption  that   $H^1(E; \scrO_E) = H^2(\hY; \scrO_{\hY}) =0$. For example, if $\hY$ is the resolution of a strict $0$-liminal Calabi-Yau threefold, then by definition $H^1(\hY; \scrO_{\hY}) =  H^1(\hY; \scrO_{\hY}) = 0$. Since  $H^2(\hY; \scrO_{\hY}) =0$,   $H^1(\hY;\Omega^1_{\hY}) \cong \Pic  \hY  \otimes \Cee$, so that the image of $H^1(\hY;T_{\hY}(-\log E))$ in $H^1(E;T_E)$ is then
$$\{\theta \in  H^1(E;T_E): \theta \smile c_1(L|E)  = 0 \text{ for all $L \in \Pic \hY$}\}.$$
This is just the Zariski tangent space to the (smooth) locus of deformations of the $K3$ surface $E$ for which all of the line bundles $L|E$ deform or equivalently for which the classes $c_1(L|E)$ remain of type $(1,1)$ for all $L \in \Pic \hY$ (see for example \cite[Theorem 3.3.11]{Sernesi}). 

For a  general $n \ge 3$,  the pairing 
$$H^1(E;T_E) \otimes H^{n-2}(E;\Omega^1_E) \to H^{n-1}(E; \scrO_E) \cong H^{n-1}(E; \omega_E)$$
is perfect. Thus, if $I$ is the image of  $H^{n-2}(\hY;\Omega^1_{\hY})$ in $H^{n-2}(E;\Omega^1_E)$, the image of $H^1(\hY;T_{\hY}(-\log E))$ in $H^1(E;T_E)$ is a vector subspace of codimension $\dim I$. More generally, given a deformation of the pair $p\colon (\widehat{\mathcal{Y}}, \mathcal{E})\to S$ over $S =\Spec A$, where $A$ is an Artin local $\Cee$-algebra, we have the corresponding deformation $p\colon \mathcal{E} \to S$. Suppose that $H^1(\hY; \scrO_{\hY}) =0$, which will hold if $\hY$ is the resolution of a strict $0$-liminal Calabi-Yau variety. A standard argument using induction on $\ell(A)$ shows that, under this assumption,  the homomorphism $\Pic \widehat{\mathcal{Y}} \to \Pic \mathcal{E} $ is injective. Then   $\omega_{\widehat{\mathcal{Y}}/S} = \scrO_{\widehat{\mathcal{Y}}}(-\mathcal{E})$ and hence $\omega_{\mathcal{E}} \cong \scrO_{\mathcal{E}}$. There is  a commutative diagram
$$\begin{CD}
  R^1p_* T_{\widehat{\mathcal{Y}}/S}(-\log \mathcal{E}) @>>> R^1p_*T_{\mathcal{E}} @>{\partial}>> R^2p_* T_{\widehat{\mathcal{Y}}/S}(-\mathcal{E}) \\
  @VVV @VVV @VVV \\
  R^1p_*\Omega_{\widehat{\mathcal{Y}}/S}^{n-1}(\log \mathcal{E}) @>>> R^1p_*\Omega^{n-2}_{\mathcal{E}} @>{\gamma}>>  R^2p_*\Omega^{n-1}_{\widehat{\mathcal{Y}}/S}.
\end{CD}$$
As before $\gamma$ is dual to the map $\iota^*\colon R^{n-2}p_*\Omega^1_{\widehat{\mathcal{Y}}/S}  \to R^{n-2}p_*\Omega^1_{\mathcal{E}} $, which gives a restriction  on the variations of Hodge structure on $H^{n-1}(E;\Cee)$ coming from deformations of the pair $(\hY, E)$. Of course, for $n=3$ the variation of Hodge structure on $H^2(\hY; \Cee)= H^1(\hY; \Omega^1_{\hY})$ is constant. For $n> 3$, it is certainly possible  that $H^{n-2}(\hY;\Omega^1_{\hY}) =0$. For instance, the examples constructed in Section~\ref {section5} have this property for $n > 3$. If the condition $H^{n-2}(\hY;\Omega^1_{\hY}) =0$ is satisfied, then the morphism  $\mathbf{Def}_{(\hY,E)} \to \mathbf{Def}_E$ is smooth.

\begin{remark} 
It is also interesting to consider the kernel of the map $H^1(\hY;T_{\hY}(-\log E)) \to H^1(E;T_E)$. Note that $H^0(E; T_E) = 0$ if $h^{0, n-2}(E) = 0$ or equivalently there does not exist a finite \'etale cover of $E$ of the form $V\times A$, where $A$ is an abelian variety. Under the assumption that $H^0(E; T_E) = 0$, by the above remarks, the kernel of the map $H^1(\hY;T_{\hY}(-\log E)) \to H^1(E;T_E)$ is then $H^1(\hY;T_{\hY}(-E))$ which is identified with $H^1(\hY; \Omega^{n-1}_{\hY} )$. 
\end{remark}

 \subsection{Good anticanonical models and normal crossing models}\label{subsection3.6} In this subsection, we consider generally the case of a strict $0$-liminal Calabi-Yau variety $Y$   with $E$ smooth and such that $\pi \colon \hY \to Y$ is a good anticanonical resolution: $K_{\hY} = \scrO_{\hY}(-E)$. In particular $K_E \cong \scrO_E$. However, most of the discussion applies equally well to the case where $(\hY, E)$ is just assumed to be a log Calabi-Yau pair. Assume in addition that there exists a smooth projective variety $W_2$ such that $E\subseteq W_2$, $K_{W_2} = \scrO_{W_2}(-E)$, and (the $d$-semistable condition) 
$$\scrO_{\hY}(E)|E \otimes \scrO_{W_2}(E)|E \cong \scrO_E.$$
Note that if $E \subseteq \hY$ is the exceptional divisor, the expectation is that $\scrO_{\hY}(-E)|E$ is ample, or close to it. If also $E$ is isomorphic to a hypersurface in a smooth projective variety $W_2$ with $K_{W_2} = \scrO_{W_2}(-E)$ and $\scrO_{W_2}(E)|E\cong \scrO_{\hY}(-E)|E$, then $E$ is an anticanonical divisor in $W_2$ and the bundle $K_{W_2}^{-1}|E$ is  then close to being ample. Thus  $W_2$ is roughly speaking close to being  a Fano manifold.  
For example, for the hypersurfaces $Y$ constructed in  \S\ref{section5},  we can take $W_2$ to be a smooth hypersurface of degree $n+1$ in $\Pee^{n+1}$ containing $E$ as a smooth hyperplane section.  

In any case, the normal crossing variety $W=\hY \cup_EW_2$ obtained by gluing $\hY$ and $W_2$ along $E$ is $d$-semistable and has trivial dualizing sheaf: $\omega_W \cong \scrO_W$. Thus, by a result of Kawamata-Namikawa \cite{KawamataNamikawa}, under some mild conditions, $W$  is smoothable. Specifically, the conditions are:  $H^{n-1}(W; \scrO_W) = 0$ (this amounts to $H^{n-1}(\hY; \scrO_{\hY}) = H^{n-1}(W_2; \scrO_{W_2}) = 0$ and $H^{n-2}(E; \scrO_E) = 0$)  and $H^{n-2}(\hY; \scrO_{\hY}) = H^{n-2}(W_2; \scrO_{W_2}) = 0$. In terms of the deformation theory of $Y$, we have the following:

\begin{proposition}\label{blowdownFano}  In the above notation, suppose that $\scrO_{W_2}(E)$ is ample, i.e.\ that $W_2$ is Fano.  If $\mathcal{W} \to \Delta$ is the (smooth) total space of a one-parameter smoothing of $W$, then the divisor $W_2$ can be blown down to give a flat family which is a deformation of $Y$. In particular, $Y$ is smoothable.
\end{proposition} 
\begin{proof} Since $\scrO_{\mathcal{W}} \cong \scrO_{\mathcal{W}}(W_1+ W_2)$, the normal bundle $\scrO_{\mathcal{W}}(W_2)$  restricts on $W_2$ to the line bundle $\scrO_{W_2}(-E)$ which is the dual of an ample line bundle. By Grauert's criterion,  there is a contraction morphism $\rho\colon \mathcal{W} \to \overline{\mathcal{W}}$, where $\overline{\mathcal{W}}$ is a normal complex space. There is an exact sequence
$$0 \to I_{W_2}^n/I_{W_2}^{n+1} \to \scrO_{(n+1)W_2} \to \scrO_{nW_2} \to 0,$$
where $ I_{W_2} = \scrO_{\mathcal{W}}(-W_2)$ and  $I_{W_2}/I_{W_2}^2=  \scrO_{W_2}(E) = K_{W_2}^{-1}$. Because $W_2$ is Fano, $H^1(W_2; K_{W_2}^{-n}) =0$ for all $n \ge 1$. 
Thus, by the formal functions theorem,  $R^1\rho_*\scrO_{\mathcal{W}}=0$. By Wahl's theory \cite[Theorem 1.4(c)]{Wahl}, the induced morphism $\overline{\mathcal{W}} \to \Delta$ is flat and gives a smoothing of $Y$.
\end{proof}

Note that, in the examples of degree $n+2$ hypersurfaces  $Y$ in Section~\ref{section5},  the hypotheses of Proposition~\ref{blowdownFano} are satisfied. Of course, since in these examples $Y$ is a hypersurface in $\Pee^{n+1}$, it is automatically smoothable and  $\mathbf{Def}_Y$ is in addition unobstructed.

As in \S\ref{subsection3.5}, we can consider deformations of $d$-semistable  normal crossing varieties $W=\hY \cup_EW_2= W_1\cup_EW_2$, say with trivial dualizing sheaf, without assuming that $E$ is exceptional in $W_1= \hY$.  As before, we would like to compare $\mathbf{Def}_W$ with $\mathbf{Def}_E$, at least on the level of tangent spaces. We begin by collecting some basic facts about $W$.

\begin{lemma}\label{normalX} With $W=\hY \cup_EW_2= W_1\cup_EW_2$ as above, let $\nu \colon \widetilde{W} = W_1 \amalg W_2 \to W$ be the normalization morphism and  let  $i\colon E  \to W$ be the inclusion. 
\begin{enumerate}
\item[\rm(i)] There is an exact sequence 
$$0 \to T^0_W \to \nu_*(T_{W_1}(-\log E) \oplus T_{W_2}(-\log E)) \to i_*T_E \to 0.$$
\item[\rm(ii)] There is an exact sequence
$$0 \to \nu_*(T_{W_1}(-E) \oplus T_{W_2}(-E)) \to  T^0_W \to i_*T_E \to 0.$$
\end{enumerate}
\end{lemma}
\begin{proof} 

\smallskip
\noindent (i) Locally $W$ can be  embedded  in $\Cee^{n+1}$ with coordinates $z_1, \dots, z_{n+1}$ as the hypersurface defined by $z_1z_2 =0$, $i=1,2$. Set $W_i = \{z_i=0\}$. From the normal bundle sequence
$$0 \to T^0_W \to T_{\Cee^{n+1}}|W \to N_{W/\Cee^{n+1}}=\mathit{Hom}(I_W/I_W^2,  \scrO_W),$$
 the kernel of $T_{\Cee^{n+1}}|W \to N_{W/\Cee^{n+1}}$ is the $\scrO_W$-submodule of $\displaystyle T_{\Cee^{n+1}}|W= \bigoplus_{i=1}^{n+1}\scrO_W\cdot  \frac{\partial}{\partial z_i}$ generated by
$$z_1\frac{\partial}{\partial z_1} , z_2\frac{\partial}{\partial z_2}, \frac{\partial}{\partial z_3}, \dots, \frac{\partial}{\partial z_{n+1}}.$$
This submodule is then more intrinsically identified with the kernel of the surjective homomorphism $\nu_*(T_{W_1}(-\log E) \oplus T_{W_2}(-\log E)) \to i_*T_E$ induced  by the difference of the two natural homomorphisms $T_{W_i}(-\log E) \to  (a_i)_*T_E$, where $a_i \colon E \to W_i$ is the inclusion.

\smallskip
\noindent (ii) In the above notation, the  image of $\nu_*(T_{W_1}(-E) \oplus T_{W_2}(-E))$ in  $\nu_*(T_{W_1}(-\log E) \oplus T_{W_2}(-\log E))$ is generated by
$$\left( z_2\frac{\partial}{\partial z_2}, 0\right), \left(z_2\frac{\partial}{\partial z_3}, 0\right) \dots, \left(z_2\frac{\partial}{\partial z_{n+1}}, 0\right), \left(0, z_1\frac{\partial}{\partial z_1}\right) , \left(0, z_1\frac{\partial}{\partial z_3}\right), \dots, \left(0, z_1\frac{\partial}{\partial z_{n+1}}\right).$$
Thus the image is contained in $T^0_W$ and the quotient of $T^0_W$ by this image is the  $\scrO_E$-module given by $\displaystyle\bigoplus_{i=3}^{n+1} \scrO_E\cdot \frac{\partial}{\partial z_i} = i_*T_E$. 
\end{proof} 

The situation is summarized by the following diagram, where $\Delta \colon  i_*T_E \to  i_*T_E \oplus  i_*T_E$ is the diagonal embedding and $\Sigma\colon  i_*T_E \oplus  i_*T_E \to  i_*T_E$ is the subtraction map:

$$\begin{tikzcd}[column sep=small]
  &  & \scriptstyle 0 \arrow[d] &  \scriptstyle 0 \arrow[d] & \\
   \scriptstyle 0\arrow[r] &\scriptstyle \nu_*(T_{W_1}(-E) \oplus T_{W_2}(-E)) \arrow[r] \arrow[d, "="] & \scriptstyle T^0_W  \arrow[r] \arrow[d] & \scriptstyle i_*T_E\arrow[r]\arrow[d, "\Delta"] & \scriptstyle 0\\
\scriptstyle   0 \arrow[r]  &\scriptstyle \nu_*(T_{W_1}(-E) \oplus T_{W_2}(-E))\arrow[r] & \scriptstyle  \nu_*(T_{W_1}(-\log E) \oplus T_{W_2}(- \log E))\arrow[r]\arrow[d] & \scriptstyle i_*T_E\oplus i_*T_E  \arrow[r]\arrow[d, "\Sigma"]  & \scriptstyle 0 \\
  &   & \scriptstyle  i_*T_E \arrow[r, "="] \arrow[d] &\scriptstyle  i_*T_E \arrow[d] &  \\ 
    & &     \scriptstyle 0 & \scriptstyle 0 &
     \end{tikzcd}$$

\begin{remark} (i) From Lemma~\ref{normalX}(i), there is an exact sequence
\begin{gather*}
H^0(T_{W_1}(-\log E)) \oplus H^0(T_{W_2}(-\log E)) \to H^0(E; T_E) \to H^1(W; T^0_W) \to \\
\to H^1(T_{W_1}(-\log E)) \oplus H^1(T_{W_2}(-\log E)) \to H^1(E; T_E).
\end{gather*}
Geometrically, this says the following: a first order  locally trivial deformation of $W$ determines first order deformations of the two pairs $(W_1,E)$ and $(W_2, E)$ which induce the same first order deformations of $E$. A first order deformation of $W$ which induces a trivial first order deformation  of the   pairs $(W_1,E)$ and $(W_2, E)$ comes from $H^0(E; T_E)$, the tangent space to $\Aut E$, i.e.\ a first order deformation of the gluing of $E\subseteq W_1$ to $E\subseteq W_2$, modulo those induced by a  first order automorphism of either $(W_1,E)$ or $(W_2, E)$.

\smallskip
\noindent (ii) The exact sequences of   Lemma~\ref{normalX}(i) and (ii)  yield  a  commutative diagram
$$\begin{tikzcd}[column sep=small]
 \scriptstyle H^1(W; T^0_W) \arrow[r]  \arrow[d]   &\scriptstyle   H^1(E;T_E)\arrow[r] \arrow[d]    &  \scriptstyle H^2(W_1; T_{W_1}(-E)) \oplus H^2(W_2; T_{W_2}(-E))\arrow[d, "="]  \\
\scriptstyle  H^1(W_1; T_{W_1}(-\log E)) \oplus H^1(W_2; T_{W_2}(-\log E))\arrow[r]  & \scriptstyle  H^1(E;T_E)\oplus  H^1(E;T_E) \arrow[r]  &\scriptstyle  H^2(W_1; T_{W_1}(-E)) \oplus H^2(W_2; T_{W_2}(-E)) .
 \end{tikzcd}$$
This gives a compatibility between the coboundary map 
$$\partial \colon H^1(E;T_E) \to H^2(W_1; T_{W_1}(-E)) \oplus H^2(W_2; T_{W_2}(-E))$$ and  $(\partial_1, \partial_2)$, where the $\partial_i\colon H^1(E;T_E) \to H^2(W_i; T_{W_i}(-E))$ are the coboundary maps defined in the previous subsection. 
\end{remark}

\section{A series of examples}\label{section5} 

In this final section, we discuss in further detail  perhaps the simplest example of $0$-liminal singularities, namely singularities of the form $(X,x)$, where $\dim X =n$ and $x$ is an  ordinary singularity of  multiplicity $n+1$. These singularities occur naturally on hypersurfaces of degree $n+2$ in $\Pee^{n+1}$, and we describe the birational geometry and deformation theory of such hypersurfaces in some detail.

\subsection{Definition and first properties}  Let $Y$ be a hypersurface of degree $n+2$ in $\Pee^{n+1}$ with an ordinary $(n+1)$-fold point $x$ and otherwise smooth. In other words, $x$ has multiplicity $n+1$ and the blowup $\hY$ of $Y$ at $x$ has a smooth exceptional divisor $E$ which is embedded as a smooth hypersurface in $\Pee^n$ of degree $n+1$. In this example, $K_{\hY} = \scrO_{\hY}(-E)$. Thus $\pi\colon \hY \to Y$ is a good anticanonical resolution of the singular point.  In particular, by Lemma~\ref{goodisstrict}, the unique singular point of $Y$ is  a $0$-liminal hypersurface singularity.  Projecting from $x$ to $\Pee^n$ defines a degree one morphism $\nu\colon \hY \to \Pee^n$ such that $\nu(E) = \overline{E}$ is a smooth hypersurface of degree $n+1$. Since $\hY$ is a blowup of $\Pee^n$, it follows easily from Remark~\ref{rmkstrict} and the fact that $Y$ is a hypersurface of degree $n+2$ in $\Pee^{n+1}$ that $Y$ is a strict $0$-liminal Calabi-Yau variety.

To reverse the above construction, let $\overline{E} \subseteq \Pee^n$ be a smooth hypersurface of degree $n+1$ (hence Calabi-Yau), let $B\subseteq \Pee^n$ be a hypersurface of degree $n+2$ such that the complete intersection $A = \overline{E}\cap B$ is smooth, let $\nu \colon \hY\to \Pee^n$ be the blowup of $\Pee^n$ along $A$ with exceptional divisor $R$, and finally set $E$ to be the proper transform of $\overline{E}$ in $\hY$. Then we have the following:

\begin{align*}
E &= \nu^*\overline{E} - R;\\
\scrO_{\hY}(E) &= \nu^*\scrO_{\Pee^n}(n+1) \otimes \scrO_{\hY}(-R);\\
K_{\hY} &= \nu^*K_{\Pee^n}\otimes \scrO_{\hY}(R)= \nu^*\scrO_{\Pee^n}(-(n+1)) \otimes \scrO_{\hY}(R)= \scrO_{\hY}(-E);\\
N_{E/\hY} &= \nu^*\scrO_{\Pee^n}(n+1) \otimes \scrO_{\hY}(-R)|E = \scrO_E(n+1) \otimes \scrO_E(-A) = \scrO_E(n+1) \otimes\scrO_E(-(n+2) ) = \scrO_E(-1).
\end{align*}

Define the line bundle $L$ on $\hY$ by:
$$L =\nu^*\scrO_{\Pee^n}(n+2) \otimes \scrO_{\hY}(-R).$$
Thus $L|E =\scrO_E$ and $L\otimes \scrO_{\hY}(-E) = \nu^*\scrO_{\Pee^n}(1)$.

From above and the exact sequence
$$0 \to L\otimes \scrO_{\hY}(-E) \to L \to L|E \to 0$$
and the fact that $H^1(\hY; \nu^*\scrO_{\Pee^n}(1)) = H^1(\Pee^n; \scrO_{\Pee^n}(1)) =0$, we see that $h^0(\hY;L) = n+2$, the linear system defined by $L$ has no base points and defines a morphism $\varphi_L \colon \hY \to \Pee^{n+1}$ which contracts $E$. From the isomorphism $H^0(\hY; L) = H^0(\Pee^n; \scrO_{\Pee^n}(n+2) \otimes I_A)$, we can identify a section of $L$ with a hypersurface in $\Pee^n$ of degree $n+2$ containing $A$. Here, if $T$ is a hypersurface of degree $n+2$ containing $A$, then $T' =\nu^*T -R$ is a section of $L$. Conversely, if $T'$ is a section of $L$, then $\nu_*T'$ is a hypersurface in $\Pee^n$ of degree $n+2$ containing $A$.  Then from the exact sequence
$$0 \to \scrO_{\hY} \to L \to L|T' \to 0,$$
it follows that $\varphi_L$ embeds $T'$ as a hypersurface of degree $n+2$ in a hyperplane in $\Pee^{n+1}$, and hence  the image of $\varphi_L$ has degree $n+2$. Explicitly, if $F$ is a homogeneous polynomial of degree $n+1$ defining a smooth  hypersurface $\overline{E}$  and $G$ is a homogeneous polynomial of degree $n+2$ defining a  hypersurface  $B$, then, for homogeneous coordinates $t_0, \dots, t_n$ on $\Pee^n$,  a basis of $H^0(\Pee^n; \scrO_{\Pee^n}(n+2) \otimes I_A)$ is given by $G, t_0F, \dots, t_nF$  and the rational map $\Pee^n \dasharrow \Pee^{n+1}$ is given by
$$t = (t_0, \dots, t_n) \mapsto (G(t),  t_0F(t), \dots, t_nF(t)).$$
Let $z_0, \dots, z_{n+1}$ be homogeneous coordinates on $\Pee^{n+1}$. Since
$$G( t_0F(t), \dots, t_nF(t)) = F(t)^{n+2}G(t_0, \dots, t_n) = F(t_0F(t), \dots, t_nF(t))G(t_0, \dots, t_n), $$
the image $Y$ of $\varphi_L$ is contained in the hypersurface defined by $z_0F(z_1, \dots, z_{n+1}) -  G(z_1, \dots, z_{n+1})$ and hence is equal to it. 

\begin{remark}\label{remark5.1}  (i)  The polynomial $G$ is only determined up to addition by a polynomial of the form $\ell F$, where $\ell$ is a linear polynomial. Replacing $G$ by $G + \ell F$ replaces $z_0$ by the linear coordinate $z_0 -\ell(z_1, \dots, z_{n+1})$.

\smallskip
  \noindent (ii) It is easy to check that, if $A$ and $\overline{E}$ are both smooth, then $Y$ has the unique singular point $(1,0, \dots, 0)$.

\smallskip
  \noindent (iii) Conversely, if $Y\subseteq \Pee^{n+1}$ is a hypersurface of  degree $n+2$ with an ordinary $(n+1)$-fold point at $(1, 0, \dots, 0)$ then $Y$ is given by an equation of the form $z_0F(z_1, \dots, z_{n+1}) -  G(z_1, \dots, z_{n+1})$ where $F$ defines a smooth hypersurface $\overline{E}$ in $\Pee^n$, and $p_0=(1, 0, \dots, 0)$ is the unique singular point of $Y$ $\iff$ the hypersurface $G=0$ meets $\overline{E}$ transversally. In this case, the projection $(z_0, z_1, \dots, z_{n+1})\mapsto (z_1, \dots, z_{n+1})$ defines a morphism $\nu_0 \colon Y - \{p_0\} \to \Pee^n$. At a point $a=(a_1, \dots, a_{n+1}) \in \Pee^n$ where $F(a)\neq 0$, the fiber $\nu_0^{-1}(a)$ is the   point $(G(a)/F(a), a_1, \dots, a_{n+1})$. If $F(a) = 0$ and $(z_0, a) \in Y$, then $G(a) =0$ as well, and the fiber $\nu_0^{-1}(a)$ is the line spanned by $p_0$ and $(0,a)$ minus the point $p_0$. This sets up an isomorphism $\Pee^n - \overline{E} \cong Y - C(\overline{E})= Y-C(A)$, where $C(\overline{E})$ is the projective cone over $\overline{E}$:
$$C(\overline{E}) = \{(z_0, \dots, z_{n+1}) \in \Pee^{n+1}:  (z_1, \dots, z_{n+1}) \in \overline{E}\},$$
and $C(A)$ is similarly defined.

\smallskip
  \noindent (iv) With $F$ as above,   a homogeneous polynomial of degree $n+1$ defining a smooth hypersurface $\overline{E}$ in $\Pee^n$, the equation $F=0$ defines an isolated weighted homogeneous singularity $X_0$ in $\Cee^{n+1}$, the affine cone over $\overline{E}$. Using the $\Cee^*$-action, if $G$ is homogeneous of degree $n+1$, the singularity defined by $F -G$ is a small deformation of the weighted homogeneous singularity and its derivative lies in the weight $1$ subspace of $H^0(X_0, T^1_{X_0})$. Moreover, this derivative is zero $\iff$ $G$ is in the degree $n+2$ graded piece of the homogeneous Jacobian ideal of $F$. It is easy to give concrete examples where the hypersurface defined by $z_0F -G$ has a unique singularity which is weighted homogeneous. One such example is the hypersurface defined by $z_0\Big(\sum_{i=1}^{n+1}z_i^{n+1}\Big) - (\sum_{i=1}^{n+1}z_i^{n+2})$. 
  
  \smallskip
  \noindent (v)  If the singularity of $z_0F -G$  is weighted homogeneous, then by Corollary~\ref{wtdcor} it satisfies the assumption of Theorem~\ref{maintheorem}(vi): the map $K' \to H^2_E( \hX; \Omega^2_{\hX}(\log E))$ is an isomorphism. It is natural to ask if this always holds. 
  \end{remark}

\subsection{Deformation theory} We begin with a straightforward calculation. 

\begin{lemma} With $Y$ as above, $\mathbb{T}^0_Y =0$, i.e.\ $\Aut Y$ is finite.
\end{lemma}
\begin{proof}   Since $\pi$ is the blowup of $Y$ at a singular point, it is an equivariant resolution (cf.\ \cite[Proposition 1.2]{BurnsWahl}). Thus, by Lemma~\ref{T0},  
$$\mathbb{T}^0_Y = H^0(Y; T^0_Y) = H^0(Y; \pi_*T_{\hY}) = H^0(\hY; T_{\hY}).$$
So it suffices to show that $H^0(\hY; T_{\hY}) =0$. We have  $H^0(\hY; T_{\hY}) = H^0(\Pee^n; \nu_*T_{\hY}).$ By a standard result (see e.g.\ \cite[Lemma 5.1]{FL22b}), the morphism $\nu\colon \hY \to \Pee^n$ yields an exact sequence
$$0 \to \nu_*T_{\hY} \to T_{\Pee^n} \to N_{A/\Pee^n} \to 0.$$
Thus it suffices to prove that the map $H^0(\Pee^n; T_{\Pee^n}) \to H^0(A; N_{A/\Pee^n})$ is injective. In fact, the map $H^0(\Pee^n; T_{\Pee^n}) \to H^0(\overline{E}; N_{\overline{E}/\Pee^n})$ is already injective as follows from the isomorphism $H^0(\Pee^n; T_{\Pee^n}) \cong  H^0(\overline{E}; T_{\Pee^n}|\overline{E})$ and the fact that $ H^0(\overline{E}; T_{\Pee^n}|\overline{E}) \to H^0(\overline{E}; N_{\overline{E}/\Pee^n})$ is injective as $H^0(\overline{E}; T_{\overline{E}}) =0$.    The restriction $H^0(\overline{E}; N_{\overline{E}/\Pee^n}) \to H^0(A; N_{\overline{E}/\Pee^n}|A)$ is injective as one sees from the exact sequence
$$0 \to \scrO_{\overline{E}}(-1) \to N_{\overline{E}/\Pee^n} = \scrO_{\overline{E}}(n+1) \to N_{\overline{E}/\Pee^n}|A \to 0.$$
Thus the composition $H^0(\Pee^n; T_{\Pee^n}) \to  H^0(A; N_{\overline{E}/\Pee^n}|A)$ is injective. On the other hand,  the composition $H^0(\Pee^n; T_{\Pee^n}) \to  H^0(A; N_{\overline{E}/\Pee^n}|A)$ is also given as a composition
$$H^0(\Pee^n; T_{\Pee^n}) \to H^0(A; N_{A/\Pee^n}) \to H^0(A; N_{\overline{E}/\Pee^n}|A).$$
Hence $H^0(\Pee^n; T_{\Pee^n}) \to H^0(A; N_{A/\Pee^n})$ is injective, and thus $ H^0(Y; T^0_Y) = H^0(\Pee^n; \nu_*T_{\hY}) =0$.
\end{proof} 

Next we recall a standard lemma about hypersurfaces in $\Pee^{n+1}$:

\begin{lemma}\label{hyperlemma}   Let $Y$ be a reduced hypersurface of degree $d$ in $\Pee^{n+1}$. 
\begin{enumerate}
\item[\rm(i)] If $n \ge 3$, the natural map $H^0(Y; N_{Y/\Pee^{n+1}}) = H^0(Y; \scrO_Y(d)) \to \mathbb{T}^1_Y$ is surjective. Hence the image of $\mathbb{T}^1_Y$ in $H^0(Y; T^1_Y)$ is the same as the image of the natural map 
$$H^0(Y; N_{Y/\Pee^{n+1}}) \to H^0(Y; N_{Y/\Pee^{n+1}}/JN_{Y/\Pee^{n+1}}),$$
where $J$ is the Jacobian ideal: at a point  $y\in Y$ where $Y$ is defined by the local equation $f=0$ in some neighborhood of $y$ in $\Pee^{n+1}$, $J $ is generated by the partial derivatives of $f$. 
\item[\rm(ii)] For $n> 3$,  $\mathbb{T}^2_Y =0$, so that the deformations of $Y$ are unobstructed. For $n =3$ and $d=5$, $\mathbb{T}^2_Y  \neq 0$, but nonetheless the deformations of $Y$ are unobstructed. 
\end{enumerate}
\end{lemma}
\begin{proof}
From the conormal exact sequence
$$0 \to I_Y/I_Y^2 \to \Omega^1_{\Pee^{n+1}}|Y \to \Omega^1_Y \to 0,$$
we get a long exact sequence
$$ \cdots \to \mathbb{T}^i_Y \to H^i(Y; T_{\Pee^{n+1}}|Y) \to H^i(Y; \scrO_Y(d)) \to \cdots $$
The result  then follows by restricting the Euler exact sequence to $Y$ and standard results about the cohomology of $\Pee^{n+1}$ and hypersurfaces in $\Pee^{n+1}$.
\end{proof}

We turn next to the picture described in \S\ref{subsection3.5}. Note that $H^1(\hY; \Omega^{n-1}_{\hY}) \cong H^0(A; \Omega^{n-2}_A)$  because $\hY$ is the blowup of $\Pee^n$ along $A$ and by standard formulas for the cohomology of a blowup (e.g.\ \cite[Theorem 7.31]{Voisinbook}). Since $H^0(E; T_E) =0$, there is an exact sequence
$$0 \to H^1(\hY;T_{\hY}(-E)) \to  H^1(\hY; T_{\hY}(-\log E)) \to \Ker \partial \to 0,$$ 
where $H^1(\hY;T_{\hY}(-E)) \cong H^1(\hY; \Omega^{n-1}_{\hY}) \cong H^0(A; \Omega^{n-2}_A)$ and $\partial$ is defined in the discussion before Proposition~\ref{prop4.4}. Since  $A$ is a divisor in the Calabi-Yau manifold $\overline{E}$,  $\Omega^{n-2}_A = K_A =\scrO_A(A)$ and we have the exact sequence
$$0 \to \scrO_{\overline{E}} \to \scrO_{\overline{E}}(A) \to \scrO_A(A) \to 0.$$
 Thus $\dim H^0(A; \Omega^{n-2}_A) = h^0(\overline{E}; \scrO_{\overline{E}}(A)) - 1$ is the dimension of the linear system $|A|$ (viewed as a divisor on $\overline{E}$). From the exact sequence
 $$0 \to \scrO_{\Pee^n}(1) \to  \scrO_{\Pee^n}(n+2) \to \scrO_{\overline{E}}(A) \to 0,$$
 it follows that
 $$\dim H^0(A; \Omega^{n-2}_A) = \dim |A| = \binom{2n+2}{n} -(n+1) -1.$$

Since $\partial$ is identified with the Gysin map $\gamma \colon H^1(E; \Omega^{n-2}_E) \to H^2(\hY; \Omega^{n-1}_{\hY})$, $\Ker \partial = \Ker \gamma = H^1_0(E; \Omega^{n-2}_E)$, say, where $H^1_0(E; \Omega^{n-2}_E)$ denotes the primitive cohomology, and thus  is equal to $H^1(E; \Omega^{n-2}_E)$ unless $n=3$.  If $n=3$, $H^2(\hY; \Omega^{n-1}_{\hY})\cong H^2(\Pee^3; \Omega^2_{\Pee^3})\oplus H^1(A; \Omega^{n-2}_A)$ which has dimension $2$, and $ H^1(E; \Omega^1_E)$ has dimension $20$ since $E$ is a $K3$ surface. Thus $\dim \Ker \gamma = 19$, the dimension of the moduli space of quartic $K3$ surfaces.  
For $n > 3$, $H^2(\hY; \Omega^{n-1}_{\hY})\cong   H^1(A; \Omega^{n-2}_A)=0$, since $A$ is a complete intersection of dimension $n-2$. Thus  $H^1(E; \Omega^{n-2}_E) \cong H^1(E; T_E)$ is the number of moduli of the Calabi-Yau hypersurface $\overline{E}$ in $\Pee^n$, namely $\displaystyle \binom{2n+1}{n+1} - (n+1)^2$. Of course, this also gives the correct answer for the number of moduli of quartic $K3$ surfaces in case $n=3$. In all cases,
$$\dim H^1(\hY; T_{\hY}(-\log E)) = \dim |A|  + \dim H^1_0(E; \Omega^{n-2}_E) =  \binom{2n+2}{n} + \binom{2n+1}{n+1} -(n^2+3n +3),$$
i.e.\ the number of moduli of the divisor $A$ on $\overline{E}$ plus the number of moduli of the hypersurface $\overline{E}$. 

We also have the local cohomology sequence of \S\ref{subsection3.2}:
$$ 0\to  H^1(\hY; \Omega^{n-1}_{\hY}(\log E)) \to  H^1(V; T_V) \to H^2_E(\hY; \Omega^{n-1}_{\hY}(\log E)) \to  H^2(\hY; \Omega^{n-1}_{\hY}(\log E)).$$
Arguments as above show that $H^2(\hY; \Omega^{n-1}_{\hY}(\log E)) =0$ if $n > 3$ and $\dim H^2(\hY; \Omega^2_{\hY}(\log E)) =1$ for $n=3$. As for the other terms, 
 $H^1(\hY; \Omega^{n-1}_{\hY}(\log E) ) \cong H^1(\hY; T_{\hY}(-\log E))$ and 
 $$\dim H^2_E(\hY; \Omega^{n-1}_{\hY}(\log E)) = b^{1, n-2} = s_1.$$  Note that $s_1$ remains constant in a $\mu = \text{const}$ deformation, so that  we can calculate $\mu$ by calculating the corresponding value for a weighted homogeneous example, for example the Fermat singularity $X_0$ defined by 
$z_1^{n+1}  + \cdots + z_{n+1}^{n+1} = 0$. Here
$$H^0(X_0, T^1_{X_0}) = \Cee[z_1, \dots, z_{n+1}]/(z_1^n, \dots, z_{n+1}^n) \cong \Cee[z_1]/(z_1^n) \otimes \cdots \otimes \Cee[ z_{n+1}]/( z_{n+1}^n).$$
Thus $\dim H^0(X_0, T^1_{X_0}) = n^{n+1}$. As an element of $H^0(X_0, T^1_{X_0})$,  $z_i$ has $\Cee^*$-weight $-(n+1)+1=-n$. 

By Proposition~\ref{wtdcase},  the dimension of  $H^2_E(\hX_0; \Omega^{n-1}_{\hX_0}(\log E))$   is the dimension of the negative weight space of $H^0(X_0, T^1_{X_0})$.  If we set $\dim H^2_E(\hX_0; \Omega^{n-1}_{\hX_0}(\log E)) = t_-$, then for $n> 3$,
\begin{equation}\label{4.00}
\binom{2n+2}{n} + \binom{2n+1}{n+1} -(n^2+3n +3) + t_- = \dim H^1(V; T_V) = \dim \mathbb{T}^1_Y = \binom{2n+3}{n+2} - (n+2)^2.
\end{equation}
The term $t_-$ is the number of ordered $(n+1)$-tuples of   integers $a_1, \dots, a_{n+1}$ with $0\le a_i \le n-1$ such that $\sum_ia_i \le n$. Without the condition that $a_i \le n-1$, this number is the same as the number of homogeneous polynomials in $z_0, \dots, z_{n+1}$ of degree $n$. Thus
$$b^{1, n-2} = t_- = \binom{2n+1}{n} - (n+1).$$
Since  $n^2+3n +3 + (n+1) = (n+2)^2$, (\ref{4.00}) is equivalent to  the identity
$$\binom{2n+2}{n} + \binom{2n+1}{n+1} +\binom{2n+1}{n} = \binom{2n+2}{n} + \binom{2n+2}{n+1} = \binom{2n+3}{n+1}= \binom{2n+3}{n+2} .$$

\begin{remark} (i) For $n=3$, because 
$$101 = 51 + 19 + 31 = \dim  H^1(\hY; T_{\hY}(-\log E)) + b^{1, n-2} ,$$
 it follows that the map $H^1(V; T_V) \to H^2_E(\hY; \Omega^2_{\hY}(\log E))$  is surjective, even though $H^2(\hY; \Omega^2_{\hY}(\log E)) $ is not zero in this case.

\smallskip
\noindent (ii) If we would like to replace the variety $Y$ by a $d$-semistable normal crossing model as in \S\ref{subsection3.6}, it is natural to consider pairs $(W_2, E)$, where $W_2$ is a smooth projective variety of dimension $n$ with $K_{W_2} = \scrO_{W_2}(-E)$, such that $E$ is an ample divisor with $N_{E/W_2} = \scrO_E(1)$. The moduli space of pairs $(W_2, E)$, where $W_2$ is a smooth hypersurface of degree $n+1$ in $\Pee^{n+1}$ and $E$ is a hyperplane section, is easily seen to be a weighted projective space: 
write the defining equation for $W_2$ as
$$F + z_0G_n + \cdots + z_0^{n+1}G_0,$$
where $F, G_i\in \Cee[z_1, \dots, z_{n+1}]$  are homogeneous of degrees $n+1$ and $i$ respectively (so that $G_0$ is a constant) and as before $F=0$ defines $\overline{E}\subseteq \Pee^n$. The coordinates $z_1, \dots, z_{n+1}$ are only well-defined up to the coordinate transformation $z_i \mapsto z_i + a_iz_0$, and hence $G_n$ is only well-defined up to the Jacobian ideal of $F$. The remaining freedom is by multiplying $z_0$ by $\lambda \in \Cee^*$, or equivalently by acting on the $(n+1)$-tuple $(G_n, \dots,   G_0)$ by the $\Cee^*$-action with weights $(-1, \dots, -n-1)$. These are exactly the $\Cee^*$-weights for the $\Cee^*$-action on the negative weight space of $T^1_{X_0, 0}$, where $X_0$ is the affine cone over $\overline{E}\subseteq \Pee^n$. The corresponding weighted projective space  is then the usual model for the negative weight deformations of $X_0$.  Geometrically, to pass from the deformations of $Y$ to those of the normal crossings model $\hY \amalg_EW_2$, we need to make a weighted blowup of the locus of equisingular deformations of $Y$, which is roughly speaking a smooth subvariety of the deformation space of codimension $t_-$. 
\end{remark}

 \begin{example}\label{finalexample} Suppose that $Y$ is defined by the polynomial  $z_0\Big(\sum_{i=1}^{n+1}z_i^{n+1}\Big) - (\sum_{i=1}^{n+1}z_i^{n+2} )$ and $x = (1,0, \dots, 0)$. Then it is easy to check as noted in Remark~\ref{remark5.1}(iv)  that
 $$ T^1_{Y,x}\cong \Cee\{z_1, \dots, z_{n+1}\}/(z_1^n, \dots, z_{n+1}^n) \cong \Cee[z_1]/(z_1^n) \otimes \cdots \otimes \Cee[ z_{n+1}]/( z_{n+1}^n).$$
  Thus
   $$H^0(Y;T^1_Y)\cong \{P\in \Cee[z_1, \dots, z_{n+1}]: \deg_{z_i}P \leq n-1 \text{ for all $i$, $1\leq i \leq n+1$}\}.$$
  By Lemma~\ref{hyperlemma}(i), the image of $\mathbb{T}^1_Y$ in $H^0(Y;T^1_Y)$ is then identified with the space of all polynomials $P\in \Cee[z_1, \dots, z_{n+1}]$ such that $\deg_{z_i}P \leq n-1$ for all $i$ and the total degree of $P$ is at most $n+2$.  For $n \ge 3$, this is a proper subspace of $H^0(Y;T^1_Y)$, and in fact the codimension of this subspace grows rather quickly. In particular,  the map $\mathbb{T}^1_Y  \to H^0(Y;T^1_Y)$ fails to be surjective. Thus, by contrast with the situation for $n=2$, for $n\ge 3$  the deformation functor $\mathbf{Def}_Y$ is never smooth over the local deformation functor.
  
  Similar examples can be constructed for generalized Fano hypersurfaces of degree $n+1$ in $\Pee^{n+1}$ or generalized Calabi-Yau hypersurfaces of degree $n+2$ with an ordinary $n$-fold point at $(1,0, \dots, 0)$. In this case the isolated singularity is a rational singularity but it is  not $1$-rational and in fact is strongly $1$-irrational in the terminology of \cite[Definition 2.6]{FL}.  In this case, 
  the map $\mathbb{T}^1_Y  \to H^0(Y;T^1_Y)$  is not surjective as soon as $n \ge 4$  in either the Fano case  or the Calabi-Yau case.
 \end{example}

\bibliography{zerolimref}
\end{document}